\documentclass[11pt]{article}
\usepackage{amsmath,amssymb,amsthm,amsfonts,amstext,amsbsy,amscd}
\usepackage[utf8]{inputenc}
\usepackage{amsfonts}
\usepackage{amssymb, amstext}
\usepackage{color}
\usepackage{mathrsfs}

\numberwithin{equation}{section}
\theoremstyle{plain}
\newtheorem{thm}{Theorem}
\newtheorem{proposition}{Proposition}
\newtheorem{corollary}{Corollary}
\newtheorem{lemma}{Lemma}
\newtheorem{assumption}{Assumption}

\theoremstyle{definition}
\newtheorem{definition}{Definition}

\theoremstyle{remark}
\newtheorem{remark}{Remark}
\newtheorem{example}{Example}

\newcommand{\II}{\mathcal{I}}
\newcommand{\JJ}{\mathcal{J}}

\allowdisplaybreaks[4]

\title{Lower bounds for invariant statistical models with applications to principal component analysis}

\author{Martin Wahl\thanks{Institut f\"{u}r Mathematik, Humboldt-Universit\"{a}t zu Berlin, Unter den Linden 6, 10099 Berlin, Germany. E-mail: martin.wahl@math.hu-berlin.de 
\newline
\textit{2010 Mathematics Subject Classifcation.} Primary 62H25; secondary 62B10, 60B20\newline
}}

\date{}

\begin{document}
\maketitle
\begin{abstract}
This paper develops nonasymptotic information inequalities for the estimation of the eigenspaces of a covariance operator. These results generalize previous lower bounds for the spiked covariance model, and they show that recent upper bounds for models with decaying eigenvalues are sharp. The proof relies on lower bound techniques based on group invariance arguments which can also deal with a variety of other statistical models.
\end{abstract}

\tableofcontents

\section{Introduction}
\subsection{Motivation}
In this paper we address the problem of deriving lower bounds for the estimation of derived parameters of the eigenspaces of a covariance operator. Motivated by principal component analysis (PCA), our main interest lies in the eigenspace of the $d$ leading eigenvalues.

In an asymptotic framework, such lower bounds can be obtained by existing results for statistical models satisfying the local asymptotic normality (LAN) condition, such as the local asymptotic minimax theorem due to H\'{a}jek \cite{MR0400513}; see also the monographs by  Ibragimov and Has'minskii \cite{MR620321} and van der Vaart \cite{MR1652247}. Extensions are presented in Koltchinskii, L\"{o}ffler and Nickl \cite{MR4065170} and Koltchinskii \cite{K17} for the special cases of estimating linear functionals of principal components and more general smooth functionals of covariance operators, respectively. Both papers study an asymptotic scenario in which the effective rank of the covariance operator is allowed to increase with the number of observations $n$, and provide exactly matching asymptotic lower bounds based on the classical van Trees inequality.

Nonasymptotic lower bounds for the estimation of the eigenspace of the $d$ leading eigenvalues have been established in Cai, Ma and Wu \cite{MR3161458,MR3334281} and Vu and Lei \cite{MR3161452} for a spiked covariance model with two different eigenvalues. This simple worst-case model can be parametrized by the Grassmann manifold, allowing to apply lower bounds under metric entropy conditions. This (Grassmann) approach has been applied to different principal subspace estimation problems and to different spiked structures (see, e.g., Cai and Zhang \cite{MR3766946} and Cai, Li and Ma \cite{CLM20} and the references therein).

In this paper we provide information inequalities for the estimation of derived parameters of the eigenspaces of a covariance operator that are nonasymptotic on the one hand and can be applied to arbitrary sequences of eigenvalues on the other. In particular, in the case of the eigenspace of the $d$ leading eigenvalues, we obtain minimax lower bounds that match recent nonasymptotic upper bounds for models with decaying eigenvalues derived in Mas and Ruymgaart \cite{MR3300524}, Rei\ss{} and Wahl \cite{ReissWahl} and Jirak and Wahl \cite{MR4052188,JW18}. To achieve this, we develop a new approach for the construction of lower bounds for statistical models equipped with an invariant group action. In particular, we establish a Bayesian version of the Chapman-Robbins inequality and a van Trees inequality in the context of equivariant statistical models. Besides principal component analysis, our lower bound techniques can deal with a variety of other statistical models. To illustrate this further, we also discuss the matrix denoising problem and the classical nonparametric density estimation problem.

The use of group invariance arguments has turned out to be crucial in multivariate statistics (see, e.g., Muirhead \cite{MR652932}, Farrell \cite{MR770934}, Eaton \cite{MR1089423} and Johnstone \cite{MR2334195}), and our approach also relies on such principles. First, our approach is based on a van Trees inequality for equivariant models, with the reference measure being the Haar measure on the orthogonal group. Second, in order to derive tight nonasymptotic lower bounds, we have to study explicit prior distributions. We propose a density based on the exponential of the trace that can be analyzed by large deviations techniques (see, e.g., Meckes \cite{MR3971582} and Hiai and Petz \cite{MR1746976}).

\subsection{Main minimax lower bound}\label{sec_main_result}
Let us present our main minimax lower bound in the context of PCA. Let $\mathcal{H}$ be a separable Hilbert space with inner product $\langle\cdot,\cdot\rangle$ and induced norm $\|\cdot\|$, and let $L_1^{+}(\mathcal{H})$ be the class of all bounded linear operators $\Sigma:\mathcal{H}\rightarrow \mathcal{H}$ that are symmetric, positive and of trace class. For $\Sigma\in L_1^{+}(\mathcal{H})$, let $\lambda_1(\Sigma)\geq\lambda_2(\Sigma)\geq \cdots>0$ be the non-increasing sequence of positive and summable eigenvalues of $\Sigma$, and let $u_1(\Sigma),u_2(\Sigma),\dots$ be a sequence of corresponding eigenvectors. We shall assume that $u_1(\Sigma),u_2(\Sigma),\dots$ form an orthonormal basis of $\mathcal{H}$.

For a fixed  sequence $\lambda_1\geq \lambda_2\geq \cdots > 0$ of positive and summable real numbers, we define the parameter class
\begin{align*}
\Theta_\lambda=\{\Sigma\in L_1^{+}(\mathcal{H}):\lambda_j(\Sigma)=\lambda_j\text{ for all }j\geq1\},
\end{align*}
consisting of all $\Sigma\in L_1^{+}(\mathcal{H})$ with spectrum $(\lambda_1,\lambda_2, \dots)$, and consider the statistical model defined by the family of Gaussian measures 
\begin{equation}\label{eq_stat_experiment}
(\mathbb{P}_\Sigma)_{\Sigma\in \Theta_\lambda},\qquad\mathbb{P}_\Sigma=\mathcal{N}(0, \Sigma)^{\otimes n},
\end{equation} 
where $\mathcal{N}(0, \Sigma)$ denotes a Gaussian measure in $\mathcal{H}$ with expectation zero and covariance operator $\Sigma$ (see, e.g., Chapter 2 in Da Prato and Zabczyk \cite{MR3236753} for some background) and $n\geq 1$ is a natural number. 
An observation in this model consists of
%This statistical model corresponds to observing 
$n$ independent $\mathcal{H}$-valued Gaussian random variables $X_1,\dots,X_n$ with expectation zero and covariance $\Sigma\in \Theta_\lambda$, and we will write $\mathbb{E}_\Sigma$ to denote expectation with respect to $X_1,\dots,X_n$ having law $\mathbb{P}_\Sigma$.

For $\mathcal{I}\subseteq \{1,2,\dots\}$, the main parameter of interest is the orthogonal projection onto the eigenspace of the eigenvalues $\{\lambda_i(\Sigma):i\in\II\}$ given by
\begin{align*}
P_{\II}(\Sigma)=\sum_{i\in\II}u_i(\Sigma)\otimes u_i(\Sigma),
\end{align*}
where $u_i(\Sigma)\otimes u_i(\Sigma)$ denotes the orthogonal projection onto the span of $u_i(\Sigma)$. Note that $P_{\II}(\Sigma)$ is uniquely defined in case that the eigenvalues with indices in $\II$ are separated from the rest of the spectrum, that is provided that $\lambda_i(\Sigma)-\lambda_{j}(\Sigma)\neq 0$ for every $i\in\II,j\notin \II$.

Our first main result provides a nonasymptotic minimax lower bound for the estimation of the parameter $P_{\II}(\Sigma)$ in the squared Hilbert-Schmidt loss.

\begin{thm}\label{thm_minimax_Hilbert} There exists an absolute constant $c\in(0,2)$ such that
%for every summable sequence $\lambda_1\geq \lambda_2\geq \cdots \geq 0$, every 
%$\mathcal{I}\subseteq \{1,2,\dots\}$ and every $n\geq 1$, we have
\begin{align*}
\inf_{\hat P}\sup_{\Sigma\in \Theta_\lambda}\mathbb{E}_\Sigma\|\hat P-P_{\II}(\Sigma)\|_{\operatorname{HS}}^2\geq c\cdot\max_{\JJ\subseteq \{1,2,\dots\}}\sum_{i\in \II\cap \JJ}\sum_{j\in \JJ\setminus \II}\Big(\frac{\lambda_i\lambda_j}{n(\lambda_i-\lambda_j)^2}\wedge \frac{1}{|\JJ|}\Big),
\end{align*}
where the infimum is taken over all estimators $\hat P=\hat P(X_1,\dots, X_n)$ with values in the space of all Hilbert-Schmidt operators on $\mathcal{H}$, $\|\cdot\|_{\operatorname{HS}}$ denotes the Hilbert-Schmidt norm and $x\wedge y=\min(x,y)$.
\end{thm} 
Theorem \ref{thm_minimax_Hilbert} is the consequence of a more general Bayes risk bound presented in Section \ref{sec_reduction_finite_dim}.
%
%Both lower bounds follow from a van Trees inequality based on group invariance arguments. In particular, we will see that the inverses of $\lambda_i\lambda_j/(n(\lambda_i-\lambda_j)^2)$ correspond to different directions of the Fisher information of the model \eqref{eq_stat_experiment}, while the relaxing term $|\JJ|$ corresponds to the Fisher information of the prior distribution, ensuring that the lower bound does not explode for vanishing gaps.

%Ignoring the $1/|\JJ|$-term, the right-hand side of the lower bound coincides for $c=2$ with the asymptotic limit achieved by the estimator $\hat P=P_{\II}(\hat\Sigma)$, with empirical covariance operator $\hat\Sigma=n^{-1}\sum_{i=1}^nX_i\otimes X_i$; see e.g. Dauxois, Pousse and Romain \cite{MR650934} and Rei\ss{} and Wahl \cite{ReissWahl}. Moreover, it follows from the local asymptotic minimax theorem that this estimator is indeed asymptotically efficient. While such asymptotic results also easily follow from our theory (cf. Corollary \ref{cor_lower_bound_SP_preliminary} below), it is exactly the difficult to obtain structure including the $1/|\JJ|$-terms that makes our bounds non-asymptotic and leads to optimal lower bounds as illustrated in the next section. 
%

\subsection{Examples and discussion}\label{sec_applications}
Let us discuss some consequences of Theorem \ref{thm_minimax_Hilbert}. 
For a subset $\II\subseteq\{1,2,\dots\}$ and a sequence of eigenvalues, we abbreviate the minimax risk over $\Theta_\lambda$ as follows
\begin{align}\label{eq_asymptotically_efficient}
R^*_{n,\lambda,\II}=\inf_{\hat P}\sup_{\Sigma\in \Theta_\lambda}\mathbb{E}_\Sigma\|\hat P-P_{\II}(\Sigma)\|_{\operatorname{HS}}^2,
\end{align}
where the infimum is taken over all estimators $\hat P=\hat P(X_1,\dots, X_n)$ taking values in the Hilbert space of all Hilbert-Schmidt operators on $\mathcal{H}$. Note that the minimax risk increases if we restrict the infimum to all estimators taking values in the class of all orthogonal projections of rank $|\II|$, in which case the Hilbert-Schmidt distance $\|\hat P-P_{\II}(\Sigma)\|_{\operatorname{HS}}$ is equal to $\sqrt{2}$ times the Euclidean norm of the sines of the canonical angles between the subspaces corresponding to $\hat P$ and $P_{\II}(\Sigma)$ (see, e.g., \cite[Chapter VII.1]{MR1477662}).

Let us start by briefly describing the classical asymptotic scenario $n\rightarrow \infty$. Combining Theorem \ref{thm_minimax_Hilbert} (resp.~the proof in Section \ref{proof_thm_lower_bound_Bayes_risk}, where it is shown that for each $\delta\in(0,2)$, the constant $c$ in Theorem \ref{thm_minimax_Hilbert} can be replaced by $2-\delta$ provided that $1/|\JJ|$ is replaced by $c_\delta/|\JJ|$ with a constant $c_\delta>0$ depending only on $\delta$) with standard perturbation bounds for the empirical covariance operator, we get
\begin{align}\label{eq_asymptotic_limit}
\lim_{n\rightarrow \infty}n\cdot R_{n,\lambda,\II}^*=2\sum_{i\in \II}\sum_{j\notin \II}\frac{\lambda_i\lambda_j}{(\lambda_i-\lambda_j)^2},
\end{align}
where the asymptotic limit is achieved by the estimator $\hat P=P_{\II}(\hat\Sigma)$ with empirical covariance operator $\hat\Sigma=n^{-1}\sum_{i=1}^nX_i\otimes X_i$ (see, e.g., Hsing and Eubank \cite{MR3379106} and Jirak and Wahl \cite[Equation (1.3)]{MR4052188} for the corresponding upper bound). More specifically, our approach also implies that $P_{\II}(\hat\Sigma)$ is asymptotically efficient in the sense of \cite[Equation (9.4)]{MR620321} (see Section \ref{sec_simple_example} and Corollary \ref{cor_lower_bound_SP_preliminary} below for more details).

Let us turn to nonasymptotic lower bounds. As stated in \cite{MR3161458}, it is highly nontrivial to obtain lower bounds which depend optimally on all parameters, in particular the eigenvalues and $\II$. Indeed, in contrast to the asymptotic setting, in which the above results can also be derived from the local asymptotic minimax theorem, it seems unavoidable to use some more sophisticated facts on the underlying parameter class of all orthonormal bases in order to obtain nonasymptotic lower bounds. A fundamental result, obtained in Cai, Ma and Wu \cite{MR3161458} and Vu and Lei \cite{MR3161452}, provides a nonasymptotic lower bound for a spiked covariance model with two groups of eigenvalues.

\begin{corollary}[\cite{MR3161458}]\label{cor_scm} For $\mathcal{H}=\mathbb{R}^p$ and the sequence of eigenvalues $\lambda_1=\dots=\lambda_d>\lambda_{d+1}=\dots=\lambda_p>0$, we have
\[
R^*_{n,\lambda,\{1,\dots,d\}}\geq c\cdot\min\Big(\frac{d(p-d)}{n}\frac{\lambda_d\lambda_{d+1}}{(\lambda_d-\lambda_{d+1})^2},d,p-d\Big)
\]
for some absolute constant $c>0$.
\end{corollary} 
Corollary \ref{cor_scm} is sharp up to a constant (see, e.g., \cite[Theorem 5]{MR3161458} for a corresponding upper bound). Together the lower and upper bound can be viewed as a nonasymptotic version of the phase transition phenomenon for empirical eigenvectors (see, e.g., \cite{P07,N08}).

Corollary \ref{cor_scm} follows from Theorem \ref{thm_minimax_Hilbert} applied with  $\JJ=\{1,\dots,p\}$, using that $2d(p-d)/p\geq \min(d,p-d)$. In \cite{MR3161458,MR3161452}, the proof of Corollary \ref{cor_scm} is based on the behavior of the local metric entropy of the Grassmann manifold. In fact, the parameter class can be written as $\Theta_\lambda=(\lambda_d-\lambda_{d+1})\mathcal{P}_d+\lambda_{d+1}I_p$, with $I_p$ being the identity matrix and $\mathcal{P}_d$ being the class of all orthogonal projections of rank $d$, for which we have the following two facts:
\begin{enumerate}
\item[(i)] For $\Sigma_i=(\lambda_d-\lambda_{d+1}) P_i+\lambda_{d+1}I_p$ with $P_i\in\mathcal{P}_d$, $i=1,2$, we have, with Kullback-Leibler divergence $K(\cdot,\cdot)$,
\begin{align*}
 K(\mathbb{P}_{\Sigma_1},\mathbb{P}_{\Sigma_2})=\frac{n(\lambda_d-\lambda_{d+1})^2}{4\lambda_d \lambda_{d+1}}\|P_1-P_2\|_{\operatorname{HS}}^2.
\end{align*}
\item[(ii)] For all integers $1\leq d\leq p$ such that $d\leq p-d$ and all real number $\delta\in(0,\sqrt{2d}]$, the covering number $ \mathcal{N}(\mathcal{P}_d,\|\cdot\|_{\operatorname{HS}},\delta)$ of $\mathcal{P}_d$ with respect to the Hilbert-Schmidt norm satisfies, 
\begin{align*}\Big(\frac{c_0\sqrt{d}}{\delta}\Big)^{d(p-d)}\leq \mathcal{N}\big(\mathcal{P}_d,\|\cdot\|_{\operatorname{HS}},\delta\big)\leq \Big(\frac{c_1\sqrt{d}}{\delta}\Big)^{d(p-d)}
\end{align*} 
for some absolute constants $c_0,c_1>0$. In particular, for any $\alpha\in(0,1)$ and $\delta\in(0,\sqrt{2d}]$ there are $P_1,\dots,P_m\in \mathcal{P}_d$ with $m\geq (c_0/(\alpha c_1 ))^{d(p-d)}$ and $\alpha\delta\leq \|P_i-P_j\|_{\operatorname{HS}}\leq 2\delta$ for every $i\neq j$.
\end{enumerate} 

For (i) see \cite[Lemma A.2]{MR3161452}, (ii) can be found in \cite[Proposition 8]{MR1665590} and \cite[Lemma 1]{MR3161458}. Combining (i) and (ii), Corollary \ref{cor_scm} follows from Fano's lemma (cf. \cite[Lemma 3]{MR1462963}).

The interesting point here is that we obtain Corollary \ref{cor_scm} by an entirely different approach via invariance arguments. This will allow us to take advantage of the Fisher geometry of the statistical model \eqref{eq_stat_experiment_version} more efficiently, and leads to the improvement of Theorem \ref{thm_minimax_Hilbert} over Corollary \ref{cor_scm}. A first extension of Corollary \ref{cor_scm} is as follows.

\begin{corollary}\label{cor_scm_gen} For $\mathcal{H}=\mathbb{R}^p$, $1\leq d\leq  p-d$ and the sequence of eigenvalues $\lambda_1\geq \dots\geq \lambda_d>\lambda_{d+1}=\dots=\lambda_p>0$, we have
\[
R^*_{n,\lambda,\{1,\dots,d\}}\geq  c\cdot\sum_{i\leq d}\min\Big( \frac{p}{n}\frac{\lambda_i\lambda_{d+1}}{(\lambda_i-\lambda_{d+1})^2},1\Big)
\]
for some absolute constant $c>0$.
\end{corollary}   
Corollary \ref{cor_scm_gen} follows from Theorem \ref{thm_minimax_Hilbert} applied with  $\JJ=\{1,\dots,p\}$, using that $p-d\geq p/2$. Let us give an interpretation of Corollary \ref{cor_scm_gen} by comparing it to upper bounds for the eigenprojections $\hat P=P_{\{1,\dots,d\}}(\hat\Sigma)$ of the empirical covariance operator $\hat\Sigma$ established in \cite{ReissWahl}. To see this, write
\begin{align*}
\|P_{\{1,\dots,d\}}(\hat\Sigma)-P_{\{1,\dots,d\}}(\Sigma)\|_{\operatorname{HS}}^2&=2\sum_{i\leq d}\sum_{j>d}\| P_i\hat P_{j}\|_{\operatorname{HS}}^2,
\end{align*}
where we abbreviated $P_i=P_{\{i\}}(\Sigma)$ and $\hat P_{j}=P_{\{j\}}(\hat\Sigma)$. Now, there are two completely different possibilities to bound the latter projector norms. First, we always have $\sum_{j>d}\| P_i\hat P_{j}\|_{\operatorname{HS}}^2\leq 1$. Second, one can apply perturbation bounds, and the first part in the minimum in Corollary~\ref{cor_scm_gen} gives the size of the linear perturbation approximation of $\sum_{j>d}\| P_i\hat P_{j}\|_{\operatorname{HS}}^2$ (see, e.g., \cite{MR650934,MR4052188}). In particular, Corollary \ref{cor_scm_gen} implies that one can not improve upon a mixture of both inequalities.

A main focus of Theorem \ref{thm_minimax_Hilbert} is on PCA in infinite dimensions, typically encountered in functional data analysis and kernel methods in machine learning. For instance, Sobolev kernels usually lead to polynomially decaying eigenvalues, while smooth radial kernels have nearly exponentially decaying eigenvalues (cf.~\cite{conf/colt/SteinwartHS09,MR2450103,DBLP:conf/colt/Belkin18} and the references therein). In such scenarios, it is, in contrast to the spiked covariance model, no longer sufficient to study the first few principal components. Instead, it becomes more important to understand the behavior of eigenspaces for growing values of $d$ (see, e.g., \cite{MR2332269} and \cite{MR3833647} for two accounts on principal component regression and more general spectral regularization methods, respectively).

\begin{corollary}\label{cor_polynomial}
For each $\alpha>0$, there is a constant $c>0$ depending only on $\alpha$ such that the following holds.
\begin{itemize}
\item[(i)] If $\lambda_j=j^{-\alpha-1}$ for every $j\geq 1$, then we have
\begin{align*}
R^*_{n,\lambda,\{d\}}&\geq  c\cdot\min\Big(1,\frac{d^2}{n}\Big),\\
R^*_{n,\lambda,\{1,\dots,d\}}&\geq c\cdot\min\Big(d,\frac{d^2}{n}\Big(1+\log_+\Big(d\wedge\sqrt{\frac{n}{d}}\Big)\Big)\Big)
\end{align*}
with $\log_+(x)=0\vee\log(x)$.
\item[(ii)] If $\lambda_j=e^{-\alpha j}$ for every $j\geq 1$, then we have 
\begin{align*}
R^*_{n,\lambda,\{d\}}\geq  \frac{c}{n},\quad R^*_{n,\lambda,\{1,\dots,d\}}\geq  \frac{c}{n}.
\end{align*}
\end{itemize}  
\end{corollary}

Corollary \ref{cor_polynomial} follows from applying Theorem \ref{thm_minimax_Hilbert} with $\JJ=\{d,d+1\}$ and $\JJ=\{j:d/2< j\leq 3d/2\}$ respectively. See Appendix \ref{sec_appendix} for the details of the second claim in (i).
In the non-trivial regime $d^2\leq n$ (note that $R^*_{n,\lambda,\{d\}}\leq 1$, as can be seen by considering the zero estimator), Corollary \ref{cor_polynomial} (i) can be written as
\begin{align}\label{eq_cor_polynomial}
R^*_{n,\lambda,\{d\}}\geq  c\frac{d^2}{n}\quad\text{and}\quad
R^*_{n,\lambda,\{1,\dots,d\}}\geq c\frac{d^2\log (d)}{n}.
\end{align}
A corresponding non-aymptotic upper bound has been first obtained in Mas and Ruymgaart \cite{MR3300524}, yet with additional $\log$ factors in $n$ and $d$ and for the smaller operator norm instead of the Hilbert-Schmidt norm. 
In case of the Hilbert-Schmidt norm, Jirak and Wahl \cite{JW18,MR4052188} showed that the empirical projectors $P_{\{d\}}(\hat\Sigma)$ and $P_{\{1,\dots,d\}}(\hat\Sigma)$ achieve, with high probability, the upper bounds $Cd^2/n$ and $Cd^2\log(d)/n$, respectively, as long as $d^2\log^2 (d)\leq cn$. This implies that the lower bounds in \eqref{eq_cor_polynomial} are, up to a constant, sharp in the latter regime. 
In case of exponentially decaying eigenvalues, matching upper bounds have been derived in Wahl \cite{Wahl} and Rei\ss{} and Wahl \cite{ReissWahl} for the case of $\II=\{d\}$ and $\II=\{1,\dots,d\}$, respectively.

We conclude this section by describing some extensions and open problems. Firstly, while our results are restricted to the Hilbert-Schmidt loss, we leave it for further research to extent our bounds to other loss functions, such as the operator norm and the excess risk loss. In particular, since the excess risk can be written as a weighted squared loss (cf.~\cite[Lemma 2.6]{ReissWahl}), lower bounds for the excess risk can be established by appropriate extensions of Proposition \ref{prop_bayesian_chapman_robbins_ineq}.  Secondly, it has been shown in \cite{MR4052188,JW18,Wahl} that for models with decaying eigenvalues, so-called relative rank conditions are crucial to characterize the behavior of empirical eigenvalues and eigenspaces. Such conditions do not necessarily follow from Theorem \ref{thm_minimax_Hilbert}. For instance, Corollary~\ref{cor_polynomial}~(ii) yields the lower bound $c/n$ for estimating $P_{\{d\}}(\Sigma)$ which is obviously sub-optimal for $d\geq n$. A second problem is therefore to establish lower bounds that provide an (optimal) phase transition beyond which estimating the eigenspaces is no longer possible. Thirdly, while we provide quantitative results that can be used to show that the PCA projector $P_{\{1,\dots,d\}}(\hat\Sigma)$ achieves the minimax risk up to a constant, it seems to be an open problem to establish corresponding qualitative properties. Is it true, for instance, that the PCA projector is even minimax (and admissible) with respect to the Hilbert-Schmidt loss over $\Theta_\lambda$, at least over all estimators taking values in the class of all orthogonal projections of rank $d$. In fact, while it is not too difficult to show this in the setting of Corollary \ref{cor_scm}, it seems to be an open problem whether the PCA projector is minimax or not in the other examples (see also \cite{R18} for a general conjecture).

\subsection{Main Bayes risk lower bound}\label{sec_reduction_finite_dim}
We now state a stronger Bayes risk lower bound for PCA in the special case of $\mathcal{H}=\mathbb{R}^p$. Afterwards, we show how Theorem \ref{thm_minimax_Hilbert} can be obtained from this result by standard decision-theoretic arguments.
In what follows, we will use the following reparametrization of our model
\begin{equation}\label{eq_stat_experiment_version}
(\mathbb{P}_U)_{U\in SO(p)},\qquad\mathbb{P}_U=\mathcal{N}(0, U\Lambda U^T)^{\otimes n},
\end{equation} 
where $SO(p)$ denotes the special orthogonal group (defined in Section \ref{sec_prel_notation} below) and $\Lambda=\operatorname{diag}(\lambda_1,\dots,\lambda_p)$ with $\lambda_1\geq \dots\geq \lambda_p> 0$ fixed. While this leads to a slight over-parametrization (each covariance matrix is now repeated at least $2^{p-1}$-times), it will allow us to invoke more directly properties of the special orthogonal group. In this reparametrization, the parameter of interest is $P_{\II}(U)=\sum_{i\in\II}u_iu_i^T$, where $u_1,\dots,u_p$ are the columns of $U\in SO(p)$. 

\begin{thm}\label{thm_lower_bound_Bayes_risk}
Consider the statistical model given in \eqref{eq_stat_experiment_version}. Then there are absolute constants $c,C>0$ such that
\begin{align*}
&\inf_{\hat P}\int_{SO(p)}\mathbb{E}_{U}\|\hat P-P_{\II}(U)\|_{\operatorname{HS}}^2\,\pi_h(U)dU\\
&\geq c\sum_{i\leq d}\sum_{j>d}\min\Big(\frac{1}{n}\frac{\lambda_i\lambda_j}{(\lambda_i-\lambda_j)^2},\frac{1}{h^2p}\Big)\qquad\forall h\geq C,
\end{align*}
where the infimum is taken over all $\mathbb{R}^{p\times p}$-valued estimators $\hat P=\hat P(X_1,\dots, X_n)$, $dU$ denotes the Haar measure on $SO(p)$, and $\pi_h$ is the prior density given~by
$$
\pi_h(U)=\frac{\exp(hp\operatorname{tr}(U))}{\int_{SO(p)}\exp(hp\operatorname{tr}(U))\,dU}
$$
with $\operatorname{tr}(U)$ denoting the trace of $U$.
\end{thm}

While $h=C$ provides an optimal choice from an non-asymptotic point of view, the choice $h=o(\sqrt{n})$ leads to a Bayesian version of the local asymptotic minimax theorem, using that the constant $c>0$ can be made more explicit, as shown in Section \ref{proof_thm_lower_bound_Bayes_risk}.

We conclude this section by showing how Theorem \ref{thm_lower_bound_Bayes_risk} with $h=C$ implies Theorem~\ref{thm_minimax_Hilbert}.  
For some fixed orthonormal basis $u_1,u_2,\dots$ of $\mathcal{H}$ and a finite subset $\JJ\subseteq \{1,2,\dots\}$ with $|\JJ|=p$, we consider the sub-model
\begin{align*}
\Theta_{\lambda}^{\JJ}=\{\Sigma\in\Theta_\lambda:\forall k\notin \JJ,P_k(\Sigma)=u_k\otimes u_k\}\subseteq \Theta_\lambda,
\end{align*}
leading to covariance operators $\Sigma\in \Theta_\lambda$ that only differ from each other on how they map $\operatorname{span}\{u_j:j\in \JJ\}$ into itself. 
%
%Consider now the easier problem of estimating $P_{\II}(\Sigma^{(p)})\in\mathbb{R}^{p\times p}$ with $\Sigma^{(p)}=(\langle u_j,\Sigma u_k\rangle)_{j,k=1}^p\in\mathbb{R}^{p\times p}$. 
Hence, restricting the Hilbert-Schmidt norm to inner products with $\{u_j:j\in\JJ\}$, we have 
\begin{align}\label{eq_simpler_problem}
\mathbb{E}_\Sigma\big[\|\hat P-P_{\II}(\Sigma)\|_{\operatorname{HS}}^2\big]\geq \mathbb{E}_\Sigma\big[\|\hat P^{\JJ}-P_{\II\cap\JJ}(\Sigma^{\JJ})\|_{\operatorname{HS}}^2\big],
\end{align}
with $\hat P^{\JJ},\Sigma^{\JJ}\in\mathbb{R}^{p\times p}$ defined by $\hat P^{\JJ}_{jk}=\langle u_j,\hat P u_k\rangle$ and $\Sigma^{\JJ}_{jk}=\langle u_j,\Sigma u_k\rangle$ for $j,k\in\JJ$. Here, we used that the matrix $P_{\II\cap\JJ}(\Sigma^{\JJ})$ has entries $\langle u_j,P_{\II}(\Sigma)u_k\rangle$, $j,k\in\JJ$.
While $\hat P^{\JJ}$ is based on $X_1,\dots,X_n$, it follows by an application of the sufficiency principle that it suffices to consider estimators based on $X_i^{\JJ}=(\langle X_i,u_j\rangle)_{j\in\JJ}$, $i\leq n$. 
Indeed, by the Neyman factorization theorem (\cite[Theorem 1.1]{MR620321} or \cite[Theorem 20.9]{MR812467}) and basic facts for Gaussian measures on separable Hilbert spaces (Chapter 2 in \cite{MR3236753}) we know that $(X_i^{\JJ}:i\leq n)$ is a sufficient statistic for $\{\mathbb{P}_{\Sigma}:\Sigma\in \Theta_{\lambda}^{\JJ}\}$. 
Hence, applying Rao-Blackwell's theorem (\cite[Theorem 2.1]{MR620321}) to the right-hand side of \eqref{eq_simpler_problem}, we conclude that 
\begin{align}\label{eq_sufficiency_principle}
\inf_{\hat P}\sup_{\Sigma\in \Theta_\lambda^{\JJ}}\mathbb{E}_\Sigma\big[\|\hat P-P_{\II}(\Sigma)\|_{\operatorname{HS}}^2\big] \geq \inf_{\tilde P}\sup_{\Sigma\in \Theta_\lambda^{\JJ}}\mathbb{E}_\Sigma\big[\|\tilde P-P_{\II\cap\JJ}(\Sigma^{\JJ})\|_{\operatorname{HS}}^2\big],
\end{align}
where the infimum is taken over all estimators $\tilde P=\tilde P(X_i^{\JJ}:i\leq n)$ with values in $\mathbb{R}^{p\times p}$. Now, the $X_i^{\JJ}$ are independent $\mathcal{N}(0,\Sigma^{\JJ})$-distributed random variables, and the class $\{\Sigma^{\JJ}:\Sigma\in\Theta_{\lambda}^{\JJ}\}$ coincides with $\{U\operatorname{diag}(\lambda_j:j\in\JJ) U^T:U\in SO(p)\}$. 
Hence, using that the maximum risk over a parameter class is bounded from below by the Bayes risk over any sub-class, Theorem~\ref{thm_minimax_Hilbert} follows from applying Theorem \ref{thm_lower_bound_Bayes_risk} to the right-hand side of \eqref{eq_sufficiency_principle}.

\subsection*{Outline}
The remainder of the paper is organized as follows. Section \ref{sec_prel_notation} provides some preliminaries on the special orthogonal group, on $\chi^2$-divergence and Fisher information, and on statistical models endowed with a group action. Section~\ref{sec_information_inequalities_group_action} develops a Bayesian Chapman-Robbins inequality and a van Trees-type inequality in the context of equivariant statistical models. Section~\ref{sec_proof_main_result} is devoted to the proof of Theorem \ref{thm_lower_bound_Bayes_risk}. Section \ref{sec_general_bound_eigenspaces} specializes our van Trees type inequality to the case of eigenspaces and Section \ref{sec_design_optimal_densities} studies the prior density from Theorem \ref{thm_lower_bound_Bayes_risk} based on large deviations techniques for the special orthogonal group. %Besides PCA, we also discuss a matrix denoising problem.

\section{Preliminaries}\label{sec_prel_notation}
\subsection{The special orthogonal group}
Let us introduce some notation in connection with the (special) orthogonal group (see, e.g., \cite{MR3971582} for a detailed account). The special orthogonal group is defined by $SO(p)=\{U\in\mathbb{R}^{p\times p}:UU^T=I_p,\det(U)=1\}$, where $I_p$ denotes the $p\times p$ identity matrix. It is a connected and compact Lie group. The Lie algebra (i.e., the tangent space at $I_p$) is given by $\mathfrak{so}(p)=\{\xi\in\mathbb{R}^{p\times p}:\xi+\xi^T=0\}$. It is generated by $L^{(ij)}=e_ie_j^T-e_je_i^T$, $i<j$, where $e_1,\dots,e_p$ denotes the standard basis in $\mathbb{R}^p$. Thus we have $\dim \mathfrak{so}(p)=p(p-1)/2$. More generally, the tangent space at $U$ is given by $U\mathfrak{so}(p)=\mathfrak{so}(p)U$. The exponential map is given by
\[
\exp:\mathfrak{so}(p)\rightarrow SO(p), \xi\mapsto \exp(\xi)=\sum_{k\geq 0}\xi^k/k!
\] 
and has the property that $(d/dt)\exp(t\xi)=\xi\exp(t\xi)=\exp(t\xi)\xi$. %Moreover, it is easy to see that $\exp(tL^{(ij)})$ is a rotation through an angle $t$, acting on the two-dimensional space spanned by the standard basis vectors $e_i$ and $e_j$. %Finally, it is a well-known fact that geodesics in $SO(p)$ have the form $U\exp(t\xi)$, $U\in SO(p)$ and $\xi\in \mathfrak{so}(p)$, and that $U\exp(t\xi)$ is of constant speed $\|\xi\|_2^2$. 
Finally, there is a unique translation-invariant probability measure $\mu$ (called Haar measure) on $SO(p)$ satisfying
\begin{equation}\label{EqDefHaarSO}
\int_{SO(p)}f(UV)\,d\mu(U)=\int_{SO(p)}f(VU)\,d\mu(U)=\int_{SO(p)}f(U)\,d\mu(U)
\end{equation}
for all $V\in SO(p)$ and all integrable functions $f$. To simplify the notation, we will denote $\int_{SO(p)}f(U)\,d\mu(U)$ by $\int_{SO(p)}f(U)\,dU$.

\subsection{$\chi^2$-divergence and Fisher information}\label{sec_fisher_information}
In this section we provide some standard Fisher information calculations for the statistical model given in \eqref{eq_stat_experiment_version} and a related matrix denoising model. For similar results and further reading see \cite{MR1698873,MR1311972} and \cite{MR3701408}.

The $\chi^2$-divergence between two probability measures $\mathbb{P}$ and $\mathbb{Q}$ (defined on the same measurable space) is defined as
\[
\chi^2(\mathbb{P},\mathbb{Q})=\begin{cases}\int\big(\frac{d\mathbb{P}}{d\mathbb{Q}}\big)^2\,d\mathbb{Q}-1,&\quad \text{if }\mathbb{P}\ll\mathbb{Q},\\ \infty,&\quad \text{otherwise}.\end{cases}
\] 
%Let us present some properties of the $\chi^2$-divergence. 
If $\mathbb{P}=\mathbb{P}_1\otimes\mathbb{P}_2$ and $\mathbb{Q}=\mathbb{Q}_1\otimes\mathbb{Q}_2$ are product measures, then we have
\begin{equation} \label{EqChiSquareP1}
\chi^2(\mathbb{P},\mathbb{Q})=(1+\chi^2(\mathbb{P}_1,\mathbb{Q}_1))(1+\chi^2(\mathbb{P}_2,\mathbb{Q}_2))-1.
\end{equation}
The following lemma analyzes the limiting behavior of the $\chi^2$-divergence of the family of probability measures from \eqref{eq_stat_experiment_version}, given by 
$\{\mathbb{P}_U:U\in SO(p)\}$ with $\mathbb{P}_U=\mathcal{N}(0, U\Lambda U^T)^{\otimes n}$, $\Lambda=\operatorname{diag}(\lambda_1,\dots,\lambda_p)$ and $\lambda_1\geq \dots\geq \lambda_p> 0$. 
\begin{lemma}\label{LemChi2Fisher}
Consider the statistical model given in \eqref{eq_stat_experiment_version}. Then, for $\xi\in\mathfrak{so}(p)$, we have
\[
\chi^2(\mathbb{P}_{\exp(t\xi)},\mathbb{P}_{I_p})/t^2\rightarrow \frac{n}{2}\sum_{i,j=1}^p\xi_{ij}^2\frac{(\lambda_i-\lambda_j)^2}{\lambda_i\lambda_j}\quad \text{as }t\rightarrow 0.
\]
%In particular, we have
%\[
%\chi^2(\mathbb{P}_{\exp(t\xi)}^{\otimes n},\mathbb{P}_{I_p}^{\otimes n})/t^2\rightarrow n\mathcal{I}_{I_p}(\xi,\xi)\quad \text{as }t\rightarrow 0.
%\]
\end{lemma}
\begin{proof}
Using \eqref{EqChiSquareP1}, it suffices to consider the case $n=1$, in which case we have $\mathbb{P}_{\exp(t\xi)}=\mathcal{N}(0,\exp(t\xi)\Lambda\exp(-t\xi))$. Then, for all $t$ sufficiently small, we have
\[
\chi^2(\mathbb{P}_{\exp(t\xi)},\mathbb{P}_{I_p})+1=1/\sqrt{\det\big(2I_p-\Lambda\exp(t\xi)\Lambda^{-1}\exp(-t\xi)\big)},
\]
as can be seen by inserting the multivariate Gaussian density function into the definition of the $\chi^2$-divergence. For all $t$ sufficiently small, set $f(t)=\det(2I_p-\Lambda\exp(t\xi)\Lambda^{-1}\exp(-t\xi))$. Then it follows from standard formulas for the derivatives of a determinant (see, e.g., \cite[Chapter 8.3]{MR1698873} and also \cite[Section 2.1.1]{IMM2012-03274}) that $f(0)=1$, $f'(0)=0$ and
\begin{align}
f''(0)=4\operatorname{tr}(\Lambda\xi\Lambda^{-1}\xi-\xi^2)&=2\sum_{i,j=1}^p\xi_{ij}\xi_{ji}(\lambda_i\lambda_j^{-1}+\lambda_j\lambda_i^{-1}-2)\nonumber\\
&=-2\sum_{i,j=1}^p\xi_{ij}^2\frac{(\lambda_i-\lambda_j)^2}{\lambda_i\lambda_j}.\label{eq_computation_derivative}
\end{align}
Using Taylor's theorem we obtain that 
\[
\chi^2(\mathbb{P}_{\exp(t\xi)},\mathbb{P}_{I_p})/t^2=\frac{\sqrt{1/f(t)}-1}{t^2}\rightarrow -\frac{f''(0)}{4}\quad \text{as }t\rightarrow 0,
\]
and the claim follows from inserting \eqref{eq_computation_derivative}.
\end{proof}
We now give a definition of the Fisher information that is sufficient for our purposes. Let $(\mathcal{X},\mathcal{F},(\mathbb{P}_\theta)_{\theta\in\Theta})$ be a statistical model, where $\Theta$ is a manifold embedded in some Euclidean space. Suppose that the experiment is dominated by a measure $\mu$ such that $f(x,\theta)=(d\mathbb{P}_\theta/d\mu)(x)$ are strictly positive. Moreover, let $l(x,\theta)=\log f(x,\theta)$. Then the Fisher information form at $\theta$ is defined by
\[
\mathcal{I}_\theta:T_\theta\Theta\times T_\theta\Theta\rightarrow \mathbb{R}, (v,w)\mapsto\int_{\mathcal{X}}dl(x,\theta)v\,dl(x,\theta)w\,d\mathbb{P}_\theta(x),
\]
provided that the last integrals exist. Here, $T_\theta\Theta$ denotes the tangent space of $\Theta$ at $\theta$ and $dl(x,\theta)v$ denotes the derivative of $l(x,\cdot)$ at $\theta$ in the direction $v$, defined by $dl(x,\theta)v= (d/dt)\,l(x,\gamma(t))|_{t=0}$ with $\gamma:(-\epsilon,\epsilon)\rightarrow \Theta$ such that $\gamma(0)=\theta$ and $\gamma'(0)=v$.

The following lemma explains the connection of the result in Lemma \ref{LemChi2Fisher} with the Fisher information. 
\begin{lemma}\label{lem_fisher_information}
Consider the statistical model given in \eqref{eq_stat_experiment_version}. Then the Fisher information form at $I_p$ is given by 
\[
\mathcal{I}_{I_p}:\mathfrak{so}(p)\times\mathfrak{so}(p)\rightarrow \mathbb{R}, (\xi,\eta)\mapsto \frac{n}{2}\sum_{i,j=1}^p\xi_{ij}\eta_{ij}\frac{(\lambda_i-\lambda_j)^2}{\lambda_i\lambda_j}.
\]
More generally, for $U\in SO(p)$, we have $\mathcal{I}_{U}(U\xi,U\eta)=\mathcal{I}_{I_p}(\xi,\eta)$.
\end{lemma}
\begin{remark}\label{rem_fisher_information_diagonalized}
In particular, $\mathcal{I}_{I_p}$ is diagonalized by the basis $\{L^{(ij)}:i<j\}$.
\end{remark}
\begin{remark}
Using the Fisher information form from Lemma \ref{lem_fisher_information}, Lemma \ref{LemChi2Fisher} can be written as $\chi^2(\mathbb{P}_{\exp(t\xi)},\mathbb{P}_{I_p})/t^2\rightarrow \mathcal{I}_{I_p}(\xi,\xi)$ as $t\rightarrow 0$. For the more general concept of $L^2$-differentiability, see, e.g., \cite[Chapter 1.8]{MR943833}.
\end{remark}
\begin{proof}[Proof of Lemma \ref{lem_fisher_information}]
Without loss of generality, we may assume that $n=1$.
It suffices to show that 
\[
\mathcal{I}_{I_p}(L^{(ij)},L^{(kl)})=\delta_{ik}\delta_{jl}\frac{(\lambda_i-\lambda_j)^2}{\lambda_i\lambda_j}\qquad\forall i<j,\forall k<l
\]
with $\delta_{ik}$ equal to $1$ if $i=k$ and equal to $0$ otherwise. Now, with $\mu$ being the Lebesgue measure on $\mathbb{R}^p$,
\begin{align*}
dl(x,I_p)L^{(ij)}&=-\frac{d}{dt}\frac{1}{2}\langle x,\exp(-tL^{(ij)})\Lambda^{-1}\exp(tL^{(ij)})x\rangle\Big|_{t=0}\\
&=\frac{1}{2}\langle x,(L^{(ij)}\Lambda^{-1}-\Lambda^{-1} L^{(ij)})x\rangle=(\lambda_j^{-1}-\lambda_i^{-1})x_ix_j.
\end{align*}
Thus
\begin{align*}
\mathcal{I}_{I_p}(L^{(ij)},L^{(kl)})&=(\lambda_j^{-1}-\lambda_i^{-1})(\lambda_l^{-1}-\lambda_k^{-1})\int x_ix_jx_kx_l\,d\mathcal{N}(0,\Lambda)(x)\\
&=\delta_{ik}\delta_{jl}(\lambda_j^{-1}-\lambda_i^{-1})^2\lambda_i\lambda_j,
\end{align*}
and the claim follows.
\end{proof}
We now provide similar calculations for a related matrix denoising model. For a fixed diagonal matrix $\Lambda=\operatorname{diag}(\lambda_1,\dots,\lambda_p)$ with $\lambda_1\geq \dots\geq \lambda_p\geq 0$, we consider the family of probability measures $(\mathbb{P}_U)_{U\in SO(p)}$ with $\mathbb{P}_U$ being the distribution of
\begin{align*}
X=U\Lambda U^T+\sigma W,
\end{align*}
where $\sigma>0$ is fixed and $W=(W_{ij})_{1\leq i,j\leq p}$ is a GOE matrix, that is a symmetric random matrix whose upper triangular entries are independent zero mean Gaussian random variables with $\mathbb{E}W_{ij}^2=1$ for $1\leq i<j\leq p$ and $\mathbb{E}W_{ii}^2=2$ for $i=1,\dots,p$ (see, e.g., \cite{MR3739989,MR3845022,MR4052188}). Using half-vectorization, this model can alternatively be defined on $\mathcal{X}=\mathbb{R}^{p(p+1)/2}$ with 
\begin{equation}\label{eq_stat_experiment_low_rank}
(\mathbb{P}_U)_{U\in O(p)},\qquad\mathbb{P}_U=\mathcal{N}(\operatorname{vech}(U\Lambda U^T), \sigma^2\Sigma_W),
\end{equation} 
where symmetric matrices $A\in\mathbb{R}^{p\times p}$ are transformed
into vectors using $\operatorname{vech}(A)=(A_{11},A_{21},\dots,A_{p1},A_{22},A_{32}\dots,A_{p2},\dots,A_{pp})\in \mathbb{R}^{p(p+1)/2}$, and $\Sigma_W$ is the covariance matrix of $\operatorname{vech}(W)$.
\begin{lemma}\label{LemChi2Fisher_low_rank}
Consider the statistical model given in \eqref{eq_stat_experiment_low_rank}. Then, for $\xi\in\mathfrak{so}(p)$, we have
\[
\chi^2(\mathbb{P}_{\exp(t\xi)},\mathbb{P}_{I_p})/t^2\rightarrow \frac{1}{2\sigma^2}\sum_{i,j=1}^p\xi_{ij}^2(\lambda_i-\lambda_j)^2\quad \text{as }t\rightarrow 0.
\]
%In particular, we have
%\[
%\chi^2(\mathbb{P}_{\exp(t\xi)}^{\otimes n},\mathbb{P}_{I_p}^{\otimes n})/t^2\rightarrow n\mathcal{I}_{I_p}(\xi,\xi)\quad \text{as }t\rightarrow 0.
%\]
\end{lemma}
\begin{proof}
Using the identity 
\begin{align*}
\chi^2(\mathcal{N}(\mu_1,\Sigma),\mathcal{N}(\mu_2,\Sigma))=\exp(\|\Sigma^{-1/2}(\mu_1-\mu_2)\|_2^2)-1
\end{align*}
with Euclidean norm $\|\cdot\|_2$, we get
\begin{align*}
\chi^2(\mathbb{P}_{\exp(t\xi)},\mathbb{P}_{I_p})=\exp\Big(\frac{1}{2\sigma^{2}}\|\exp(t\xi)\Lambda\exp(-t\xi)-\Lambda\|_{\operatorname{HS}}^2\Big)-1,
\end{align*}
where we also used the identity $\|\Sigma_W^{-1/2}\operatorname{vech}(A)\|_2^2=2^{-1}\|A\|_{\operatorname{HS}}^2$, valid for any $A\in\mathbb{R}^{p\times p}$ symmetric, in order to reverse the half-vectorization.
From this, we get
\begin{align*}
\chi^2(\mathbb{P}_{\exp(t\xi)},\mathbb{P}_{I_p})/t^2\rightarrow \frac{1}{2\sigma^2}\|\xi\Lambda-\Lambda\xi\|_{\operatorname{HS}}^2\quad \text{as }t\rightarrow 0,
\end{align*}
and the claim follows.
\end{proof}
\begin{lemma}\label{lem_fisher_information_low_rank}
Consider the statistical model given in \eqref{eq_stat_experiment_low_rank}. Then the Fisher information form at $I_p$ is given by 
\[
\mathcal{I}_{I_p}:\mathfrak{so}(p)\times\mathfrak{so}(p)\rightarrow \mathbb{R}, (\xi,\eta)\mapsto \frac{1}{2\sigma^2}\sum_{i,j=1}^p\xi_{ij}\eta_{ij}(\lambda_i-\lambda_j)^2.
\]
%More generally, for $U\in SO(p)$, we have $\mathcal{I}_{U}(U\xi,U\eta)=\mathcal{I}_{I_p}(\xi,\eta)$.
\end{lemma}
\begin{proof}
A similar calculation as in the proof of Lemma \ref{LemChi2Fisher_low_rank} shows that $\mathcal{I}_{I_p}(\xi,\eta)=(2\sigma^2)^{-1}\operatorname{tr}((\xi\Lambda-\Lambda\xi)(\eta\Lambda-\Lambda\eta)^T)$, $\xi,\eta\in\mathfrak{so}(p)$, and the claim follows.
\end{proof}

\subsection{Statistical models under group action}\label{sec_group_action} 
In this section we summarize some basic facts for equivariant statistical models (see, e.g., \cite{MR1089423,MR2431769} for two detailed accounts). Let $G$ be a group. For every $g\in G$ the right multiplication map $R_g:G\rightarrow G$ is defined by $R_g(h)=hg$, $h\in G$. If $G$ acts (from the left) on a measurable space $(\mathcal{X},\mathcal{F})$, then we will always assume that the map $\mathcal{X}\rightarrow \mathcal{X}, x\mapsto gx$ is measurable for every $g\in G$. If $G$ itself is also a measurable space, then we will always assume that the map $G\times\mathcal{X}\rightarrow \mathcal{X}, (g,x)\mapsto gx$ is measurable.
\begin{definition}\label{def_equivariant}
Suppose that $G$ acts on $(\mathcal{X},\mathcal{F})$ and $\Theta$. Then a family $(\mathbb{P}_\theta)_{\theta\in\Theta}$ of probability measures on $(\mathcal{X},\mathcal{F})$ is called $G$-equivariant if
\[
\mathbb{P}_{g\theta}(gA)=\mathbb{P}_\theta(A)\qquad\forall g\in G, \theta\in\Theta,A\in\mathcal{F}.
\]
\end{definition} 
Definition \ref{def_equivariant} says that for a random variable $X$ with distribution $\mathbb{P}_\theta$, the random variable $gX$ has distribution $\mathbb{P}_{g\theta}$.
\begin{lemma}\label{LemPropGEquiv}
Suppose that $(\mathbb{P}_\theta)_{\theta\in\Theta}$ is $G$-equivariant. For $\theta_0,\theta_1\in\Theta$ and $g\in G$ the following holds: 
\begin{itemize}
\item[(i)] For all measurable functions $f\geq 0$, we have
\[
\int_{\mathcal{X}} f(x)\,d\mathbb{P}_{g\theta_0}(x)=\int_{\mathcal{X}} f(gx)\,d\mathbb{P}_{\theta_0}(x).
\]
\item[(ii)] If $\mathbb{P}_{\theta_1}\ll\mathbb{P}_{\theta_0}$, then $\mathbb{P}_{g\theta_1}\ll\mathbb{P}_{g\theta_0}$ and 
\[
\frac{d\mathbb{P}_{g\theta_1}}{d\mathbb{P}_{g\theta_0}}(gx)=\frac{d\mathbb{P}_{\theta_1}}{d\mathbb{P}_{\theta_0}}(x)\qquad \mathbb{P}_{\theta_0}\text{-a.e. } x.
\] 
\item[(iii)] We have $\chi^2(\mathbb{P}_{g\theta_1},\mathbb{P}_{g\theta_0})=\chi^2(\mathbb{P}_{\theta_1},\mathbb{P}_{\theta_0})$.
\end{itemize} 
\end{lemma}
Claim (i) follows from Definition \ref{def_equivariant} and standard measure-theoretic arguments, (ii) is a consequence of (i), and (iii) is a consequence of~(ii).

Given an action $G$ on a measurable space $(\mathcal{X},\mathcal{F})$, there is an induced action on the set of all probability measures on $(\mathcal{X},\mathcal{F})$ defined by $g\mathbb{P}(A)=\mathbb{P}(g^{-1}A)$, $g\in G$, $\mathbb{P}$ probability measure on $(\mathcal{X},\mathcal{F})$, and $A\in\mathcal{F}$. Using this action, a family of probability measures $\mathcal{P}$ on $(\mathcal{X},\mathcal{F})$  is called $G$-invariant if for each $\mathbb{P}\in\mathcal{P}$, we have $g\mathbb{P}\in \mathcal{P}$ for all $g\in G$. On the other hand, the term equivariant is used when the group actions on sample space and parameter class lead to the specific relation given in Definition~\ref{def_equivariant}. A more general definition that also applies to estimators is as follows.

\begin{definition}
Suppose that $G$ acts on both $\mathcal{X}$ and $\mathcal{Y}$. Then a function $\psi:\mathcal{X}\rightarrow\mathcal{Y}$ is called $G$-equivariant if $\psi(gx)=g\psi(x)$ for every $g\in G$, $x\in \mathcal{X}$.
\end{definition} 
The following lemma collects some properties on the minimax and Bayes risk of equivariant statistical models.
Let $(\mathcal{X},\mathcal{F},(\mathbb{P}_\theta)_{\theta\in\Theta})$ be a statistical model and $\psi:\Theta\rightarrow \mathbb{R}^m$ be a derived parameter. For a loss function $L:\mathbb{R}^m\times \mathbb{R}^m\rightarrow [0,\infty)$ and an estimator $\hat\psi:\mathcal{X}\rightarrow \mathbb{R}^m$, the risk function is given by $\mathbb{E}_\theta L(\hat\psi(X),\psi(\theta))=\int L(\hat\psi(x),\psi(\theta))\,d\mathbb{P}_\theta(x)$, where $X$ is an observation from the model with distribution $\mathbb{P}_\theta$.
\begin{lemma}\label{LemConstantRisk}
Suppose that $(\mathbb{P}_\theta)_{\theta\in\Theta}$ and $\psi$ are $G$-equivariant and that the loss function $L$ is convex in the first argument and satisfies $L(ga,gb)=L(a,b)$ for every $g\in G$, $a,b\in\mathbb{R}^m$. Then, for any $G$-equivariant estimator $\hat\psi$, 
\[
\mathbb{E}_\theta L(\hat\psi(X),\psi(\theta))=\mathbb{E}_{g\theta} L(\hat\psi(X),\psi(g\theta))\qquad\forall
\theta\in\Theta,g\in G.
\]
Suppose additionally that $G$ is a compact group with Haar measure $\mu$. Then, for any estimator $\tilde\psi$, the estimator $\hat\psi$ defined by $\hat\psi(x)=\int_Gg^{-1}\tilde\psi(gx)\,\mu(dg)$, $x\in\mathcal{X}$ (provided that it exists) is $G$-equivariant with
\begin{align*}
\mathbb{E}_\theta L(\hat\psi(X),\psi(\theta))\leq \int_G \mathbb{E}_{g\theta} L(\tilde\psi(X),\psi(g\theta)) d\mu(g) %\leq \sup_{g\in G}\mathbb{E}_{g\theta} L(\tilde\psi(X),\psi(g\theta))
\qquad\forall \theta\in\Theta.
\end{align*}
\end{lemma}

Let us conclude this section by showing how the two statistical models discussed in Section \ref{sec_fisher_information} can be realized as equivariant models.

\begin{example}\label{ex_PCA_equivariant}
For $\mathcal{X}=(\mathbb{R}^{p})^n$ equipped with its Borel $\sigma$-algebra $\mathscr{B}_{\mathbb{R}^p}^{\otimes n}$, consider the family of probability distributions $\{\mathbb{P}_U:U\in SO(p)\}$ given in~\eqref{eq_stat_experiment_version} with $\mathbb{P}_U=\mathcal{N}(0, U\Lambda U^T)^{\otimes n}$, $\Lambda=\operatorname{diag}(\lambda_1,\dots,\lambda_p)$ and $\lambda_1\geq  \cdots \geq \lambda_p> 0$. In this case we have that if $(X_1,\dots,X_n)$ has distribution $\mathbb{P}_U$, then $(VX_1,\dots,VX_n)$ has distribution $\mathbb{P}_{VU}$, $U,V\in SO(p)$. Hence, letting $SO(p)$ act coordinate-wise on $(\mathbb{R}^p)^n$, the latter property translates to $\mathbb{P}_{VU}(VA)=\mathbb{P}_U(A)$ for every $U,V\in SO(p)$ and every $A\in \mathscr{B}_{\mathbb{R}^p}^{\otimes n}$, meaning that the family is $SO(p)$-equivariant.
\end{example}

\begin{example}\label{ex_matrix_denoising_equivariant}
For $\mathcal{X}=\mathbb{R}^{p\times p}$ equipped with its Borel $\sigma$-algebra $\mathscr{B}_{\mathbb{R}^{p\times p}}$, consider the family of probability distributions $\{\mathbb{P}_U:U\in SO(p)\}$ with $\mathbb{P}_U$ being the distribution of $X=U\Lambda U^T+\sigma W$, where $\sigma>0$, $\Lambda=\operatorname{diag}(\lambda_1,\dots,\lambda_p)$ with $\lambda_1\geq  \cdots \geq \lambda_p\geq 0$ and $W$ is a GOE matrix. Since the GOE ensemble is invariant under orthogonal conjugation, the random variable $X=U\Lambda U^T+\sigma W$ has the property that $VXV^T$ is equal in distribution to $VU\Lambda (VU)^T+\sigma W$, $U,V\in SO(p)$. Hence, letting $SO(p)$ act on $\mathbb{R}^{p\times p}$ by conjugation, the latter property translates to the fact that the family is $SO(p)$-equivariant.
\end{example}

\section{Information inequalities under group action}\label{sec_information_inequalities_group_action}
\subsection{An equivariant Chapman-Robbins inequality}
The one-dimensional Chapman-Robbins inequality is a simple lower bound for the variance of an unbiased estimator of a real-valued parameter. Letting $(\mathcal{X},\mathcal{F},(\mathbb{P}_\theta)_{\theta\in\Theta})$ be a statistical model, $\psi:\Theta\rightarrow \mathbb{R}^m$ be a derived parameter, $\hat\psi:\mathcal{X}\rightarrow \mathbb{R}^m$ be an unbiased estimator (i.e., $\mathbb{E}_\theta \hat\psi(X)=\psi(\theta)$ for all $\theta\in\Theta$ with $X$ being an observation from the model), a multidimensional version says that 
\begin{align}\label{eq_chapman_robbins_multi}
\mathbb{E}_{\theta}\|\hat\psi(X)-\psi(\theta)\|_2^2\geq \frac{\big(\sum_{j=1}^m(\psi_j(\theta_j)-\psi_j(\theta))\big)^2}{\sum_{j=1}^m\chi^2(\mathbb{P}_{\theta_j},\mathbb{P}_{\theta})}
\end{align}
for every $\theta,\theta_1,\dots,\theta_m\in\Theta$ and $\|\cdot\|_2$ being the Euclidean norm. The proof is simple and based on the identity
\begin{align}\label{eq_chapman_robbins}
\sum_{j=1}^m\int_\mathcal{X}(\hat\psi_j(x)-\psi_j(\theta))(d\mathbb{P}_{\theta_j}(x)-d\mathbb{P}_{\theta}(x))=\sum_{j=1}^m(\psi_j(\theta_j)-\psi_j(\theta))
\end{align}
in combination with the Cauchy-Schwarz inequality (applied twice). 
%The idea is to transfer a difference in probability measures into a difference in the corresponding parameters.
In the case that $\Theta\subseteq\mathbb{R}^d$ and under additional regularity conditions (e.g., the differentiability of $\psi$ at $\theta$ and the $L^2$-differentiability of the model), \eqref{eq_chapman_robbins_multi} implies the classical Cramér-Rao lower bound (in the case $m>1$ the quadratic risk corresponds to the trace of the covariance matrix of the estimator, and we get the Cramér-Rao lower bound when the trace is applied to both sides). This can be easily seen by setting $\theta_j=\theta+tv_j$, $t\rightarrow 0$, and optimizing in the $v_j$ (see, e.g., \cite{R14} for the case $m=1$).  

If $\hat\psi$ is biased, then the above approach does not work anymore. Based on a variation of \eqref{eq_chapman_robbins}, the following proposition provides a Bayesian version of the Chapman-Robbins inequality for equivariant statistical models.
%in the sense of Definition \ref{def_equivariant}.
%

\begin{proposition}\label{prop_bayesian_chapman_robbins_ineq}
Let $(\mathcal{X},\mathcal{F},(\mathbb{P}_\theta)_{\theta\in\Theta})$ be a statistical model. Let $G$ be a topological group acting on $(\mathcal{X},\mathcal{F})$ and $\Theta$ such that $(\mathbb{P}_\theta)_{\theta\in\Theta}$ is $G$-equivariant. Suppose that $g\mapsto \mathbb{P}_{g\theta}(A)$ is measurable for every $A\in\mathcal{F}$, $\theta\in\Theta$. Let $\Pi$ be a Borel probability measure on $G$. Let $\psi:\Theta\rightarrow \mathbb{R}^m$ be a derived parameter such that $\int_G\|\psi(g\theta)\|^2d\Pi(g)<\infty$ for all $\theta\in\Theta$ and let $\hat\psi:\mathcal{X}\rightarrow \mathbb{R}^m$ be an estimator. Then, for each $\theta\in\Theta$ and each $h_1,\dots,h_m\in G$, we have  
%\begin{align*}
%&\int_{G}\mathbb{E}_{g\theta}\big[\|\hat\psi(X)-\psi(g\theta)\|_2^2\big]\,d\Pi(g)\\
%&\geq \frac{\big( \sum_{j=1}^m\int_G(\psi_j(gh_j^{-1}\theta)-\psi_j(g\theta))\,d\Pi(g)\big)^2}{\sum_{j=1}^m\big((\chi^2(\mathbb{P}_{h_j\theta},\mathbb{P}_\theta)+1)(\chi^2(\Pi\circ R_{h_j},\Pi)+1)-1\big)}.
%\end{align*}
\begin{align*}
&\int_{G}\mathbb{E}_{g\theta}\|\hat\psi(X)-\psi(g\theta)\|_2^2\,d\Pi(g)\\
&\geq \frac{\big( \sum_{j=1}^m\int_G(\psi_j(gh_j^{-1}\theta)-\psi_j(g\theta))\,d\Pi(g)\big)^2}{\sum_{j=1}^m\big(\chi^2(\mathbb{P}_{h_j\theta},\mathbb{P}_\theta)+\chi^2(\Pi\circ R_{h_j},\Pi)+\chi^2(\mathbb{P}_{h_j\theta},\mathbb{P}_\theta)\chi^2(\Pi\circ R_{h_j},\Pi)\big)},
\end{align*}
wuth $X$ being an observation from the model and $\Pi\circ R_{h_j}$ defined by $\Pi\circ R_{h_j}(B)=\Pi(Bh_j)$ for every Borel set $B$ in $G$.
\end{proposition}
\begin{remark}
If $G$ is locally compact, then it is natural to choose a probability density function $\pi$ with respect to the right Haar measure $\mu$ on $G$. Then we have $\Pi\circ R_{h}(A)=\int_{Ah}\pi(g)\,d\mu(g)=\int_A\pi(gh)\,d\mu(g)$. Note that the choice $\pi\equiv 1$ leads to the trivial lower bound zero, meaning that in applications one has to construct non-uniform prior densities.
\end{remark}
%
%\begin{remark}
%Similar considerations are also possible for statistical models without group action. For instance suppose that there is a sequence of parameters $\theta_0,\dots,\theta_n$, $n\geq 1$ and a sequence of real numbers $\pi_1,\dots,\pi_n> 0$ with $\sum_{k=1}^n\pi_k=1$ such that $\mathbb{P}_{\theta_0}=\mathbb{P}_{\theta_n}$ and $\psi(\theta_0)=\psi(\theta_n)$ holds. Then we have
%\[
%\sum_{k=1}^n\int_\mathcal{X}(\hat\psi(x)-\psi(\theta_k))(\pi_kd\mathbb{P}_{\theta_k}(x)-\pi_{k-1}d\mathbb{P}_{\theta_{k-1}}(x))=\sum_{k=1}^n(\psi(\theta_k)-\psi(\theta_{k-1}))\pi_{k-1},
%\]
%with $\pi_0=\pi_n$, as can be seen from a telescoping argument and partial summation. Proceeding similarly as in the proof of Proposition \ref{prop_bayesian_chapman_robbins_ineq} this leads to a lower bound for the Bayes risk $\sum_{k=1}^n\mathbb{E}_{\theta_{k}}[(\hat\psi-\psi(\theta_k))^2]\pi_k$, from which, under additional regularity conditions, it is possible to obtain a version of the classical van Trees inequality with the parameter space being the circle.
%\end{remark}
%

\begin{proof}[Proof of Proposition \ref{prop_bayesian_chapman_robbins_ineq}]
We may assume that $\mathbb{P}_{h_j\theta}\ll\mathbb{P}_{\theta}$ and $\Pi\circ R_{h_j}\ll\Pi$ for every $j=1,\dots,m$ and that the Bayes risk of $\hat\psi$ is finite because otherwise the result is trivial. Consider the expression
\begin{equation}\label{EqBasicIdentityIB}
\sum_{j=1}^m\int_G\int_{\mathcal{X}}(\hat\psi_j(x)-\psi_j(g\theta))(d\mathbb{P}_{g\theta}(x)d\Pi(g)-d\mathbb{P}_{gh_j\theta}(x)d\Pi\circ R_{h_j}(g))
\end{equation}
Applying the transformation formula (\cite[Theorem 4.1.11]{MR1932358}), the terms involving $\hat\psi_j$ cancel each other, and \eqref{EqBasicIdentityIB} is equal to 
\[
-\sum_{j=1}^m\int_G\psi_j(g\theta)(d\Pi(g)-d\Pi\circ R_{h_j}(g))=\sum_{j=1}^m\int_G(\psi_j(gh_j^{-1}\theta)-\psi_j(g\theta))\,d\Pi(g).
\] 
On the other hand, by the Cauchy-Schwarz inequality, the absolute value of \eqref{EqBasicIdentityIB} is bounded by
\begin{align*}
\sum_{j=1}^m&\Big(\int_G\int_{\mathcal{X}}(\hat\psi_j(x)-\psi_j(g\theta))^2\,d\mathbb{P}_{g\theta}(x)d\Pi(g)\Big)^{1/2}\\
\times&\Big(\int_G\int_{\mathcal{X}}\Big(1-\frac{d\mathbb{P}_{gh_j\theta}}{d\mathbb{P}_{g\theta}}(x)\frac{d\Pi\circ R_{h_j}}{d\Pi}(g)\Big)^2\,d\mathbb{P}_{g\theta}(x)d\Pi(g)\Big)^{1/2}.
\end{align*}
Since $(\mathbb{P}_\theta)_{\theta\in\Theta}$ is $G$-equivariant, Lemma \ref{LemPropGEquiv} (ii) yields that the last term is equal to
\begin{align*}
\sum_{j=1}^m&\Big(\int_G\int_{\mathcal{X}}(\hat\psi_j(x)-\psi_j(g\theta))^2\,d\mathbb{P}_{g\theta}(x)d\Pi(g)\Big)^{1/2}\\
\times&\Big(\int_G\int_{\mathcal{X}}\Big(1-\frac{d\mathbb{P}_{h_j\theta}}{d\mathbb{P}_{\theta}}(x)\frac{d\Pi\circ R_{h_j}}{d\Pi}(g)\Big)^2\,d\mathbb{P}_{\theta}(x)d\Pi(g)\Big)^{1/2}.
\end{align*}
The term in the second brackets is a $\chi^2$ divergence of two product measures. Thus, by \eqref{EqChiSquareP1}, the last term is equal to
\begin{align*}
\sum_{j=1}^m&\Big(\int_G\int_{\mathcal{X}}(\hat\psi_j(x)-\psi_j(g\theta))^2\,d\mathbb{P}_{g\theta}(x)d\Pi(g)\Big)^{1/2}\\
\times&\Big((\chi^2(\mathbb{P}_{h_j\theta},\mathbb{P}_\theta)+1)(\chi^2(\Pi\circ R_{h_j},\Pi)+1)-1\Big)^{1/2}.
\end{align*}
Applying the Cauchy-Schwarz inequality in $\mathbb{R}^m$, this is bounded by 
\begin{align*}
\Big(\int_{G}&\mathbb{E}_{g\theta}\|\hat\psi(X)-\psi(g\theta)\|_2^2\,d\Pi(g)\Big)^{1/2}\\
\times&\Big(\sum_{j=1}^m\big((\chi^2(\mathbb{P}_{h_j\theta},\mathbb{P}_\theta)+1)(\chi^2(\Pi\circ R_{h_j},\Pi)+1)-1\big)\Big)^{1/2},
\end{align*}
and the claim follows.
\end{proof}

\subsection{An equivariant van Trees inequality}
Under additional regularity conditions, Proposition \ref{prop_bayesian_chapman_robbins_ineq} implies a Bayesian version of the Cramér-Rao inequality (similarly as the classical Chapman-Robbins inequality implies the Cramér-Rao inequality). Yet, since Proposition~\ref{prop_bayesian_chapman_robbins_ineq} involves integrals with respect to the prior distribution, this requires arguments on the differentiation of integrals. Such justifications are more simple in the case of compact groups and in this section, we illustrate this for the special case where both $\Theta$ and $G$ coincide with the special orthogonal group $SO(p)$. While similar results can be formulated in other scenarios (cf.~\cite{MR3862091,GaoZhang}), this allows us to cover our two motivating examples from Section \ref{sec_group_action} on principal component analysis and the matrix denoising problem. More precisely, we consider a statistical model $(\mathcal{X},\mathcal{F},(\mathbb{P}_U)_{U\in SO(p)})$ satisfying the following two assumptions.
\begin{assumption}\label{ass_equivariant}
The family $\{\mathbb{P}_U:U\in SO(p)\}$ is $SO(p)$-equivariant (i.e., $SO(p)$ acts on $\mathcal{X}$ such that $\mathbb{P}_{VU}(VA)=\mathbb{P}_{U}(A)$ for all $U,V\in SO(p)$ and all $A\in\mathcal{F}$).
\end{assumption}
\begin{assumption}\label{ass_chi_square} There is a bilinear form $\mathcal{I}_{I_p}:\mathfrak{so}(p)\times \mathfrak{so}(p)\rightarrow \mathbb{R}$ such that for all $\xi\in\mathfrak{so}(p)$,
\begin{align}\label{eq:Fisher:info}
\chi^2(\mathbb{P}_{\exp(t\xi)},\mathbb{P}_{I_p})/t^2\rightarrow \mathcal{I}_{I_p}(\xi,\xi)\quad \text{as }t\rightarrow 0.
\end{align}
\end{assumption}
Assumption \ref{ass_equivariant} implies that the family $\{\mathbb{P}_g:g\in G\}$ is $G$-equivariant for any closed subgroup $G$ of $SO(p)$, and our goal is to apply Proposition \ref{prop_bayesian_chapman_robbins_ineq} combined with a limiting argument in this scenario.

\begin{proposition}\label{prop_van_Trees_abstract}
Let $(\mathcal{X},\mathcal{F},(\mathbb{P}_U)_{U\in SO(p)})$ be a statistical model such that Assumptions \ref{ass_equivariant} and \ref{ass_chi_square} are satisfied. Let $G$ be a closed subgroup of $SO(p)$ with Lie algebra $\mathfrak{g}$. Let $\psi:G\rightarrow \mathbb{R}^m$ be a continuous function, $\pi:G\rightarrow [0,\infty)$ be a continuous probability density function with respect to the Haar measure $dg$ on $G$ and $\xi_1,\dots,\xi_m\in\mathfrak{g}$. Suppose that $\psi$ and $\pi$ are differentiable with bounded derivatives in the sense that
\begin{align*}
\forall g\in G,\forall j\in\{1,\dots,m\},\qquad |d\psi_j(g)g\xi_j|,|d\pi(g)g\xi_j|\leq M
\end{align*}
for some constant $M>0$. Then, for any estimator $\hat\psi=\hat\psi(X_1,\dots,X_n)$ taking values in $\mathbb{R}^m$, we have
\begin{align*}
\int_G \mathbb{E}_g\|\hat\psi-\psi(g)\|_{2}^2\,\pi(g)dg\geq \frac{\big( \int_G\sum_{j=1}^md\psi_{j}(g)g\xi_j\,\pi(g)dg\big)^2}{\sum_{j=1}^m \big(\mathcal{I}_{I_p}(\xi_j,\xi_j)+\int_G\frac{(d\pi(g)g\xi_j)^2}{\pi(g)}\,dg\big)},
\end{align*}
where $\mathcal{I}_{I_p}(\cdot,\cdot)$ is the Fisher information form given in Lemma \ref{lem_fisher_information} and $\|\cdot\|_{2}$ denotes the Euclidean norm in $\mathbb{R}^m$.
\end{proposition}

\begin{remark}
Here, $d\psi_{j}(g)g\xi_j$ and $d\pi(g)g\xi_j$ denote the derivatives of $\psi_j$ and $\pi$ at the point $g$ in the direction $g\xi_j$ (cf. Section \ref{sec_fisher_information}).
\end{remark}

\begin{remark}
The information inequality involves only the Fisher information at $I_p$. This is due to the equivariance of the statistical model in \eqref{eq_stat_experiment_version}, leading to constant Fisher information form.
\end{remark}

\begin{remark}
Proposition \ref{prop_van_Trees_abstract} provides a van Trees type inequality in the context of equivariant statistical models. The van Trees inequality is a well-known lower bound technique that has been applied in a variety of problems (see, e.g., Tsybakov \cite{MR2724359} and the references therein for the classical one-dimensional inequality, Gill and Levit \cite{MR1354456} for multidimensional extensions, and Jupp \cite{MR2651957} for the case of smooth loss functions on manifolds).
\end{remark}

%\begin{remark}
%In contrast to previous multi-dimensional versions of the van Trees inequality, we  circumvent working with explicit parametrizations of the special orthogonal group. The latter seems to be a very difficult task, since common parametrizations lead to non-constant Fisher-information.
%\end{remark}

\begin{proof}[Proof of Proposition \ref{prop_van_Trees_abstract}]
Proposition \ref{prop_van_Trees_abstract} follows from Proposition \ref{prop_bayesian_chapman_robbins_ineq} applied with $h_j=\exp(t\xi_j)$, $t\rightarrow 0$, by standard measure-theoretic arguments on the differentiation of integrals where the integrand depends on a real parameter (cf. \cite[Corollary 5.9]{MR1312157}). Without loss of generality we may restrict ourselves to estimators with bounded Hilbert-Schmidt norm $\sup_{x\in \mathcal{X}}\|\hat\psi(x)\|_{\operatorname{HS}}<\infty$, such that the risk of the estimator is bounded. 

We first assume that $\pi(g)>0$ for all $g\in G$, which implies that that $\min _{g\in G}\pi(g)>0$, using that $G$ is compact and $\pi$ is continuous.
 By Assumption \ref{ass_chi_square}, we have, for every $j=1,\dots,m$,
\[
\chi^2(\mathbb{P}_{\exp(t\xi_j)},\mathbb{P}_{I_p})/t^2\rightarrow \mathcal{I}_{I_p}(\xi_j,\xi_j)\quad \text{as }t\rightarrow 0.
\]
Next, consider the term
\[
\chi^2(\Pi\circ R_{\exp(t\xi_j)},\Pi)/t^2=\int_{G}\frac{(\pi(g\exp(t\xi_j))-\pi(g))^2}{\pi(g)t^2}\,dg.
\]
By assumption, the function $s\mapsto \pi(g\exp(s\xi_j))$ is differentiable with derivative $s\mapsto d\pi(g\exp(s\xi_j))g\exp(s\xi_j)\xi_j$. By the boundedness assumption this derivative is bounded in $g\in G$ and $s\in\mathbb{R}$. Thus, we get that the difference quotient
\[
\frac{\pi(g\exp(t\xi_j))-\pi(g)}{t}
\]
is bounded in $g\in G$ and $t\in\mathbb{R}$. Moreover, it converges to $d\pi(g)g\xi_j$ as $t\rightarrow 0$. Thus the dominated convergence theorem (noting that $\pi$ is bounded away from zero) implies that for every  $j=1,\dots,m$, 
\[
\int_{G}\limits\frac{(\pi(g\exp(t\xi_j))-\pi(g))^2}{\pi(g)t^2}\,dg\rightarrow \int_G\frac{(d\pi(g)g\xi_j)^2}{\pi(g)}\,dg\quad \text{as }t\rightarrow 0.
\]
Similarly, using the boundedness assumption on the derivatives $d\psi_j$ this time, we get, for every  $j=1,\dots,m$,
\begin{align*}
&\int_{G} \frac{\psi_{j}(g\exp(-t\xi_j))-\psi_{j}(g)}{t}\,\pi(g)dg\rightarrow -\int_Gd\psi_{j}(g)g\xi_j\,\pi(g)dg\quad  \text{as } t\rightarrow 0.
\end{align*} 
Finally, the third term in the denominator in Proposition \ref{prop_bayesian_chapman_robbins_ineq} divided by $t^2$ vanishes as $t\rightarrow 0$.
Hence, for positive $\pi$, the claim follows from applying Proposition \ref{prop_bayesian_chapman_robbins_ineq} with $h_j=\exp(t\xi_j)$ and letting $t\rightarrow 0$.

It remains to consider the case that $\pi$ is not necessarily positive. Then we can consider $\pi_\epsilon=(\pi+\epsilon)/(1+\epsilon)$, for which the lower bound in Proposition~\ref{prop_van_Trees_abstract} holds by what we have shown so far. Letting $\epsilon$ go to zero, the Bayes risk with respect to the prior $\pi_\epsilon$ converges to the Bayes risk with respect to the prior $\pi$, and by the monotone convergence theorem, we have 
\begin{align*}
\int_G\frac{(d\pi_\epsilon(g)g\xi_j)^2}{\pi_\epsilon(g)}\,dg=\frac{1}{1+\epsilon}\int_G\frac{(d\pi(g)g\xi_j)^2}{\pi(g)+\epsilon}\,dg\rightarrow \int_G\frac{(d\pi(g)g\xi_j)^2}{\pi(g)}\,dg
\end{align*}
as $\epsilon\searrow 0$. The numerator is treated similarly.
\end{proof}

If $\pi$ is radially symmetric around $I_p$ (i.e., does only depend on the Hilbert-Schmidt distance $\|I_p-g\|_{\operatorname{HS}}^2=2p-2\operatorname{tr} (g)$), then we can write $\pi(g)=\tilde\pi(\operatorname{tr} (g))$ for some function $\tilde\pi:[-q,q]\rightarrow [0,\infty)$. 

\begin{proposition}\label{prop_van_Trees_abstract_version}
Let $(\mathcal{X},\mathcal{F},(\mathbb{P}_U)_{U\in SO(p)})$ be a statistical model such that Assumptions \ref{ass_equivariant} and \ref{ass_chi_square} are satisfied. Let $G$ be a closed subgroup of $SO(p)$ with Lie algebra $\mathfrak{g}$. Let $\psi:G\rightarrow \mathbb{R}^m$ be a continuous function and $\pi:G\rightarrow [0,\infty)$ be a probability density function with respect to the Haar measure $dg$ on $G$ of the form $\pi(g)=\tilde\pi(\operatorname{tr}(g))$ with $\tilde \pi:[-p,p]\rightarrow [0,\infty)$ continuously differentiable, and let $\xi_1,\dots,\xi_m\in\mathfrak{g}$. Suppose that for some $M>0$, $|d\psi_j(g)g\xi_j|\leq M$ for every $g\in G$ and every $j=1,\dots,p$. Then, for any estimator $\hat\psi=\hat\psi(X_1,\dots,X_n)$ taking values in $\mathbb{R}^m$, we have
\begin{align*}
\int_G \mathbb{E}_g\|\hat\psi-\psi(g)\|_{2}^2\,\tilde\pi(\operatorname{tr}(g))dg\geq \frac{\big( \int_G\sum_{j=1}^md\psi_{j}(g)g\xi_j\,\tilde\pi(\operatorname{tr}(g))dg\big)^2}{\sum_{j=1}^m \big(\mathcal{I}_{I_p}(\xi_j,\xi_j)+\int_G\frac{(\tilde\pi'(\operatorname{tr}(g))\operatorname{tr}(g\xi_j))^2}{\tilde\pi(\operatorname{tr}(g))}\,dg\big)}.
\end{align*}
\end{proposition}
\begin{remark}
The condition that $\tilde\pi$ is continuously differentiable can be weakened to $\tilde\pi$ absolutely continuous with bounded derivative $\tilde \pi'$. 
\end{remark}
\begin{proof}[Proof of Proposition \ref{prop_van_Trees_abstract_version}]
We have to check that $\pi$ from Proposition \ref{prop_van_Trees_abstract_version} satisfies the condition  of Proposition \ref{prop_van_Trees_abstract}. 
By definition, for $g\in G$ and $j=1,\dots,m$, we have $d\pi(g)g\xi_j=f'(0)$ with $f(t)= \tilde\pi(\operatorname{tr} (g\exp(t\xi_j))$. Hence,
\begin{align}\label{eq_van_Trees_abstract_version_derivative}
d\pi(g)g\xi_j=\operatorname{tr}(g\xi_j)\tilde\pi'(\operatorname{tr} (g)).
\end{align}
From the assumptions it follows that $\tilde\pi'$ is bounded in absolute value (say, by $C>0$). Using this and the Cauchy-Schwarz inequality, we conclude that 
\begin{align*}
\sup_{g\in G}|d\pi(g)g\xi_j|\leq C\sqrt{p}\|\xi_j\|_{\operatorname{HS}}
\end{align*}
Hence the assumptions of Proposition \ref{prop_van_Trees_abstract} are satisfied and the claim follows from Proposition \ref{prop_van_Trees_abstract_version} and \eqref{eq_van_Trees_abstract_version_derivative}.
\end{proof}
We conclude this section by providing a matrix representation of our lower bound. To achieve this, suppose that $\pi$ has finite Fisher information form
\begin{align*}
\mathcal{I}_\pi:\mathfrak{g}\times \mathfrak{g}\rightarrow \mathbb{R}, (\xi,\eta)\mapsto\int_G\frac{d\pi(g)g\xi\,d\pi(g)g\eta}{\pi(g)}\,dg
\end{align*}
satisfying the assumptions of Proposition \ref{prop_van_Trees_abstract}.
For $d=\dim \mathfrak{g}$ and a basis $L_1,\dots,L_d$ of $\mathfrak{g}$, we consider the following matrix representations 
\begin{align*}
M_{\mathcal{I}}=(\mathcal{I}_{I_p}(L_k,L_l))_{1\leq k,l\leq d},\qquad M_{\mathcal{I}_\pi}=(\mathcal{I}_\pi(L_k,L_l))_{1\leq k,l\leq d},
\end{align*} 
and
\begin{align*}
M_{d\psi,\pi}=\int_G(d\psi(g)gL_1,\dots,d\psi(g)gL_d)\,\pi(g)dg\in\mathbb{R}^{m\times d}.
\end{align*}
Then, applying Proposition \ref{prop_van_Trees_abstract} with the choices $\xi_j=\sum_{k=1}^dx_{jk}L_k$ $x_j=(x_{jk})_{1\leq k\leq d}\in\mathbb{R}^d$, $1\leq j\leq m$, the right-hand side of the lower bound can be written as
\begin{align*}
\frac{\big(\sum_{j=1}^me_j^TM_{d\psi,\pi}x_j\big)^2}{\sum_{j=1}^m x_j^T(M_{\mathcal{I}}+M_{\mathcal{I}_\pi})x_j},
\end{align*}
where $e_1,\dots,e_m$ denotes the standard basis in $\mathbb{R}^m$.
Optimizing in the $x_j$ leads to the choices $x_j=(M_{\mathcal{I}}+M_{\mathcal{I}_\pi})^{-1}M_{d\psi,\pi}^Te_j$ and to the lower bound
\begin{align*}
\int_G \mathbb{E}_g\|\hat\psi-\psi(g)\|_{2}^2\,\pi(g)dg\geq \operatorname{tr}\big(M_{d\psi,\pi}(M_{\mathcal{I}}+M_{\mathcal{I}_\pi})^{-1}M_{d\psi,\pi}^T\big).
\end{align*}
This provides a representation of our equivariant van Trees inequality that is closer to the classical Cramér-Rao inequality briefly discussed before Proposition \ref{prop_bayesian_chapman_robbins_ineq}. In Section \ref{sec_general_bound_eigenspaces}, we will use the above optimization strategy in the case of estimating the eigenspaces using the (diagonalizing) basis $L^{(ij)}$ (cf. Lemma \ref{lem_fisher_information}). The main remaining difficulty will be to construct an optimal prior density. 

\subsection{Two simple applications}
In this section, we provide two simple applications of the Bayesian Chapman-Robbins inequality and the van Trees inequality for equivariant models.

\subsubsection{Linear functionals of principal components}\label{sec_simple_example}
Let us first illustrate Proposition \ref{prop_van_Trees_abstract} in the simple case of functionals (i.e., one-dimensional derived parameters). This case can also be treated with the classical one-dimensional van Trees inequality (mainly due to the existence of simple one-parameter subgroups of the form $G=\{\exp(t\xi):t\in\mathbb{R}\}$, $\xi\in \mathfrak{so}(p)$). The full strength of our approach becomes apparent in the next section by considering high-dimensional eigenprojections.

Consider the statistical model $\{\mathbb{P}_U:U\in SO(p)\}$ from \eqref{eq_stat_experiment_version} with $\mathbb{P}_U=\mathcal{N}(0,U\Lambda U^T)^{\otimes n}$, $\Lambda=\operatorname{diag}(\lambda_1,\dots,\lambda_p)$ and $\lambda_1\geq  \cdots \geq \lambda_p> 0$. For simplicity let us consider the problem of estimating linear functionals of the principal components (cf.~\cite{MR4065170}), the general case can be treated similarly. More precisely, for $1\leq i\leq p$ and a fixed and known $\alpha\in \mathbb{R}^p$, the parameter of interest is $\psi(U)=\langle u_i,\alpha\rangle$, $U\in SO(p)$, with $u_i=Ue_i$ being the $i$-th column of $U$. To obtain non-trivial bounds we assume that $\min(\lambda_{i-1}-\lambda_{i},\lambda_i-\lambda_{i+1})>0$. Moreover, since eigenvectors can only be estimated up to a sign in general, we will consider parameter classes for which the sign of $u_i$ is uniquely determined. 

For a fixed $V\in SO(p)$, a subgroup of the form $G=\{\exp(t\xi):t\in\mathbb{R}\}$ with $\xi\in \mathfrak{so}(p)$ and a probability density function $\pi$ on $G$ with respect to the Haar measure $dg$ on $G$, we consider the Bayes risk
\begin{align*}
R^*_{V,\xi,\pi}(\alpha)=\inf_{\hat \psi}\int_{G}\mathbb{E}_{Vg}(\hat\psi-\langle Vge_i,\alpha\rangle)^2\,\pi(g)dg,
\end{align*}
where the infimum is taken over all estimators $\hat \psi=\hat \psi(X_1,\dots, X_n)$ with values in $\mathbb{R}$.
Since $R^*_{V,\xi,\pi}(\alpha)=R^*_{I_p,\xi,\pi}(V^T\alpha)$, let us focus on the case $V=I_p$. Moreover, since we consider the $i$-th eigenvector it turns out to be sufficient to consider
%\begin{align*}
%\xi=\begin{pmatrix}
% 0    & x_2 & \cdots &x_p\\
% -x_2 & 0   & \cdots & 0\\
% \vdots &\vdots & & \vdots\\
% -x_p & 0   & \cdots & 0
%\end{pmatrix},\qquad \sum_{j=2}^px_j^2=1.
%\end{align*}
\begin{align*}
\xi=\sum_{j\neq i}x_jL^{(ij)},\qquad  \sum_{j\neq i}x_j^2=1.
\end{align*}
The normalization ensures that $\exp((t+2\pi)\xi)=\exp(t\xi)$. In particular, in this parametrization, the Haar measure on $G$ is given by $\mathbf{1}_{[-\pi,\pi]}(t)dt/(2\pi)$ and it is shown in Appendix \ref{sec_appendix} that
\begin{align}\label{eq_matrix_exp}
\exp(t\xi)_{ii}=\cos (t),\quad \text{while}\quad \exp(t\xi)_{ji}=-x_{j}\sin (t)\quad \forall j\neq i,
\end{align}  
meaning that $\langle\exp(t\xi)_i,\alpha\rangle=\alpha_i\cos (t)-\sum_{j\neq i}\alpha_jx_j\sin (t)$.
In addition,  for $k\geq 1$, we choose 
\begin{align*}
\pi_k(\exp(t\xi))=4k\mathbf{1}_{[-\frac{\pi}{2k},\frac{\pi}{2k}]}(t)\cos^2 (kt).
\end{align*} 
By Example \ref{ex_PCA_equivariant} and Lemma \ref{LemChi2Fisher}, we know that the  statistical model in \eqref{eq_stat_experiment_version} satisfies Assumptions~\ref{ass_equivariant} and \ref{ass_chi_square}. Hence, applying Proposition \ref{prop_van_Trees_abstract} with the above choices, using also that the directional derivatives coincide with the usual derivative with respect to $t$, we get
\begin{align*}
&R^*_{I_p,\xi,\pi}(\alpha)=\inf_{\hat \psi}\int_{-\frac{\pi}{2k}}^{\frac{\pi}{2k}}\mathbb{E}_{\exp(t\xi)}(\hat\psi-\langle\exp(t\xi)_i,\alpha\rangle)^2\,4k\cos^2 (kt)\frac{dt}{2\pi}\\
&\geq \frac{\big( \sum_{j\neq i}\alpha_j x_j\int_{-\frac{\pi}{2k}}^{\frac{\pi}{2k}}4k\cos^2 (kt)\cos (t)\, \frac{dt}{2\pi}\big)^2}{\sum_{j\neq i} x_j^2\Big(\frac{n(\lambda_i-\lambda_j)^2}{\lambda_i\lambda_j}+\int_{-\frac{\pi}{2k}}^{\frac{\pi}{2k}}4k^3\sin^2 (kt)\,\frac{dt}{2\pi}\Big)}\geq \frac{\big( \cos(\pi/(2k))\sum_{j\neq i}\alpha_j x_j\big)^2}{\sum_{j\neq i} x_j^2\Big(\frac{n(\lambda_i-\lambda_j)^2}{\lambda_i\lambda_j}+k^2\Big)}.
\end{align*}
Optimizing in the $x_j$ we conclude that
\begin{align*}
R^*_{I_p,\xi,\pi_k}(\alpha)\geq\cos^2(\pi/(2k))\sum_{j\neq i}\alpha_j^2\Big(\frac{n(\lambda_i-\lambda_j)^2}{\lambda_i\lambda_j}+k^2\Big)^{-1}.
\end{align*}
In particular, if we choose $k_n=(\pi/2)(\sqrt{n}/c)$ with $c>0$ and let $n\rightarrow\infty$, then we get
\begin{align}\label{eq_Bayes_one_dimensional}
\liminf_{n\rightarrow\infty}n\cdot R^*_{I_p,\xi,\pi_{k_n}}(\alpha)\geq \sum_{j\neq i}\alpha_j^2\Big(\frac{(\lambda_i-\lambda_j)^2}{\lambda_i\lambda_j}+\frac{\pi^2}{4c^2}\Big)^{-1}.
\end{align} 
We thus obtain a slightly more precise version of a similar result derived in \cite{MR4065170}. One  advantage of our approach is that we directly perturb the eigenspaces, leading to lower bounds that are already quite precise for finite samples (in contrast, \cite{MR4065170,K17} consider additive perturbations of the form $\Sigma+tH$, $t\in(-\delta_n,\delta_n)$).
Clearly, \eqref{eq_Bayes_one_dimensional} also gives a lower bound for the minimax risk over the support of $\pi_k$. Let us show that \eqref{eq_Bayes_one_dimensional} implies a local asymptotic minimax theorem (cf. \cite[Theorem 8.11]{MR1652247}). Setting
\begin{align*}
\Theta_{\delta}=\Theta_{\delta}(\Lambda,\xi)=\{\Sigma=\exp(t\xi)\Lambda \exp(-t\xi):t\in (-\delta,\delta)\},
\end{align*}
it follows from \eqref{eq_Bayes_one_dimensional}, using the fact that the prior $\pi_{k_n}$ has support $[-c/\sqrt{n},c/\sqrt{n}]$,
\begin{align*}
\lim_{c\rightarrow\infty}\liminf_{n\rightarrow\infty}\inf_{\hat \psi}\sup_{\Sigma\in\Theta_{c/\sqrt{n}}}n\mathbb{E}_{\Sigma}(\hat\psi-\langle u_i(\Sigma),\alpha\rangle)^2\geq \sum_{j\neq i}\alpha_j^2 \frac{\lambda_i\lambda_j}{(\lambda_i-\lambda_j)^2},
\end{align*}
where we translated \eqref{eq_Bayes_one_dimensional}  to the notation of Section \ref{sec_main_result}. See \cite{MR4065170} for a matching upper bound. 

\subsubsection{Nonparametric density estimation}
The main focus of the equivariant Chapman-Robbins inequality in Proposition \ref{prop_bayesian_chapman_robbins_ineq} is on statistical models with group valued parameters. Yet, it is also possible to apply Proposition \ref{prop_bayesian_chapman_robbins_ineq} to statistical models not directly related to groups. In this section, we illustrate this in the simple case of density estimation. For simplicity, we only explain how Proposition \ref{prop_bayesian_chapman_robbins_ineq} implies the standard non-parametric $n^{-\beta/(2\beta+1)}$-rate for the pointwise risk over Hölder balls of smoothness $\beta>0$, but similar considerations (mentioned briefly at the end of this section) also yield corresponding results for the $L^2$ risk. For $g\in \{\pm 1\}$, we consider
\begin{align*}
f_g(x)=\frac{1}{2}+c_0h^{\beta}g\Big\{K\Big(\frac{x-1/2}{h}\Big)-K\Big(\frac{x+1/2}{h}\Big)\Big\},\qquad x\in[-1,1],
\end{align*}
where $h\in(0,1]$, $\beta>0$, $K$ is a symmetric and continuous probability density function with respect to the Lebesgue measure with support in $[-1/2,1/2]$ and $K(0)>0$, and $c_0>0$ is a sufficiently small constant such that $f_g(x)\geq 1/4$ for all $x\in[-1,1]$, $h\in(0,1]$ and $g\in\{\pm 1\}$. This provides a slight variation of a standard lower bound construction (see, e.g., \cite{MR1652247,MR2724359}). By construction, $f_{1}$ and $f_{-1}$ are probability densities on $[-1,1]$. Moreover, it is possible to choose the kernel such that $f_{1}$ and $f_{-1}$  are contained in a Hölder ball on $[-1,1]$ of smoothness $\beta$ for all $h\in(0,1]$ (see \cite[Chapter 2.5]{MR2724359} and \cite[Chapter 24]{MR1652247}).

We now consider the statistical model $([-1,1]^n,\mathscr{B}_{[-1,1]}^{\otimes n},(\mathbb{P}_g)_{g\in \{\pm 1\}})$ with $\mathbb{P}_g$ being the probability measure associated with a sample of $n$ independent random variables $X_1,\dots,X_n$ when the density of $X_i$ is $f_{g}$. By construction, we have $f_g(-x)=f_{-g}(x)$ for all $x\in[-1,1]$ and $g\in\{\pm 1\}$, where we used the symmetry of $K$. Hence, letting the group $G=\{\pm 1\}$ act on $[-1,1]^n$ coordinate-wise by multiplication (sign change), the transformation formula implies that for the random variable $X_i$ with density $f_g$, the random variable $-X_i$ has density $f_{-g}$, meaning that this statistical model is indeed $G$-equivariant. In order to apply Proposition~\ref{prop_bayesian_chapman_robbins_ineq}, we choose $h_1=-1$ and the prior $\Pi$ given by $\Pi(1)=1-\Pi(-1)=q\in(0,1)$. With these choices, we have
\begin{align}\label{eq_chi_square_comp_density}
\chi^2(\Pi\circ R_{-1},\Pi)=\frac{(1-2q)^2}{q(1-q)},\qquad \chi^2(\mathbb{P}_{-1},\mathbb{P}_1)\leq e^{2^5c_0^2h^{2\beta+1}n\|K\|_{L^2}^2},
\end{align} 
see Appendix \ref{sec_appendix} for the standard calculations. Now, considering the derived parameter $\psi(g)=f_g(1/2)$, $g\in\{\pm 1\}$, we have
%\begin{align*}
%\psi(hg)-\psi(g)=f_{-g}(1/2)-f_g(1/2)=\begin{cases}-2h^\beta c_0K(0),&\quad g=1,\\+2h^\beta c_0K(0),&\quad g=-1,\\
%\end{cases}
%\end{align*}
\begin{align*}
\sum_{g\in\{\pm 1\}}(\psi(-g)-\psi(g))\Pi(g)&=\sum_{g\in\{\pm 1\}}(f_{-g}(1/2)-f_g(1/2))\Pi(g)\\
&=-(q-(1-q))2h^\beta c_0K(0).
\end{align*}
Choosing $q=3/4$ (that is a non-uniform prior) and $h=n^{-1/(2\beta+1)}$, we obtain that
\begin{align*}
\inf_{\hat{f}}\Big\{\frac{1}{2}\mathbb{E}_1(\hat{T}-f_1(1/2))^2+\frac{1}{2}\mathbb{E}_{-1}(\hat{T}-f_{-1}(1/2))^2\Big\}\geq cn^{-\frac{2\beta}{2\beta+1}},
\end{align*}
where $c>0$ is an absolute constant and where the infimum is taken over all estimators $\hat T=\hat T(X_1,\dots,X_n)$ with values in $\mathbb{R}$. %As explained above, this lower bound extends to the minimax risk over Hölder classes with smoothness $\beta$. 

Finally, in order to deal with estimation in more than one point or the $L^2$ risk, one can consider for $g\in G=\{\pm 1\}^m$,
\begin{align*}
f_g(x)=\frac{1}{2}+c_0h^{\beta}\sum_{k=1}^mg_k\Big\{K\Big(\frac{x-x_k}{h}\Big)-K\Big(\frac{x+x_k}{h}\Big)\Big\},\qquad x\in[-1,1]
\end{align*}
with $x_k=(2k-1)/(2m)$, $k=1,\dots,m$, $h\in(0,1/m]$ and $K$ from above. The key observation is that the resulting family of probability measures $\{\mathbb{P}_g:g\in G\}$ can again be realized as an equivariant statistical model. For this, let $G$ act on $[-k/m,-(k-1)/m)\cup ((k-1)/m,k/m]$ by multiplication with $g_k$, $k=2,\dots,m$, and on $[-1/m,1/m]$ by multiplication with $g_1$. %In particular, it is again possible to apply Proposition \ref{prop_bayesian_chapman_robbins_ineq}, for instance by choosing $\Pi$ as the probability measure obtained when $g_1,\dots,g_m$ are independent and identical distributed with $\mathbb{P}(g_k=1)=1-\mathbb{P}(g_k=-1)=q\in(0,1)$ and $h_1=(-1,1,\dots,1),\dots,h_m=(1,\dots,1,-1)\in G$ where the identity element $(1,\dots,1)$ is only changed at one coordinate.

\section{Lower bounds for the estimation of eigenspaces}
\label{sec_proof_main_result}
%\section{Proof of Theorem \ref{thm_lower_bound_Bayes_risk}}
%
%The purpose of this section is to apply Proposition \ref{prop_van_Trees_abstract_version} in the case of PCA and the matrix denoising model.

\subsection{Invoking the van Trees-type inequality}\label{sec_general_bound_eigenspaces}
We first specialize Proposition \ref{prop_van_Trees_abstract_version} to the case of eigenspaces, that is to the case where
\begin{align*}
\psi(U)=P_{\II}(U)=\sum_{i\in\II}(Ue_i)(Ue_i)^T,\qquad U\in SO(p).
\end{align*} 
where $e_1,\dots,e_p$ denotes the standard basis in $\mathbb{R}^p$ and $\II\subseteq \{1,\dots,p\}$.
The following corollary applies Proposition \ref{prop_van_Trees_abstract_version} in the above setting and the choice $G=SO(p)$, and proposes a density for which the Fisher information of $\pi$ becomes tractable. 

\begin{corollary}\label{cor_lower_bound_SP_preliminary}
Let $(\mathcal{X},\mathcal{F},(\mathbb{P}_U)_{U\in SO(p)})$ be a statistical model such that Assumptions \ref{ass_equivariant} and \ref{ass_chi_square} are satisfied. Let $\tilde\pi:[-p,p]\rightarrow [0,\infty)$ be continuously differentiable with $\int_{SO(p)}\tilde\pi(\operatorname{tr} (U))\,dU=1$. Then, for any estimator $\hat P=\hat P(X_1,\dots,X_n)$ with values in $\mathbb{R}^{p\times p}$, we have
\begin{align*}
&\int_{SO(p)} \mathbb{E}_{U}\|\hat P-P_{\II}(U)\|_{\operatorname{HS}}^2\,\tilde\pi(\operatorname{tr} (U))dU\\
&\geq  2\Big(\int_{SO(p)}(U_{11}U_{22}+U_{12}U_{21})\tilde\pi(\operatorname{tr} (U))\,dU\Big)^2\\
&\ \ \ \times\sum_{i\in \II}\sum_{j\notin\II}\Big(\mathcal{I}_{I_p}(L^{(ij)},L^{(ij)})+\int_{SO(p)} (U_{12}-U_{21})^2\frac{(\tilde\pi'(\operatorname{tr} (U)))^2}{\tilde\pi(\operatorname{tr} (U))}\,dU\Big)^{-1}.
\end{align*} 
In particular, if $\tilde\pi=\tilde\pi_h$ is given by $\tilde\pi_h(\cdot)=\exp(hp\cdot)/Z_h$ with $h>0$ and normalizing constant $Z_{h}=\int_{SO(p)}\exp(hp\operatorname{tr}(U))\,dU$, then we have 
\begin{align*}
&\int_{SO(p)} \mathbb{E}_{U}\|\hat P-P_{\II}(U)\|_{\operatorname{HS}}^2\,\tilde\pi_h(\operatorname{tr} (U))dU\\
&\geq  2\Big(\int_{SO(p)}(U_{11}U_{22}+U_{12}U_{21})\tilde\pi_h(\operatorname{tr} (U))\,dU\Big)^2\\
&\ \ \ \times\sum_{i\in \II}\sum_{j\notin \II}\Big(\mathcal{I}_{I_p}(L^{(ij)},L^{(ij)})+8h^2p\Big)^{-1}.
\end{align*} 
\end{corollary}

\begin{remark}
Obviously, the two terms involving the prior density are competing: if the prior $\tilde\pi$  approximates the Dirac measure at $I_p$, then the prefactor tends to one, but the Fisher information of the prior explodes at the same time. On the other hand, if we choose $\tilde\pi\equiv 1$, then the Fisher information of the prior is zero, but the average in the prefactor as well. 
\end{remark}

%\begin{remark}\label{rem_Dirac}
%From Lemma \ref{LemConstantRisk} it follows that the minimax risk is determined by estimators that have constant risk. This explains why it is sufficient to construct a measure that is radially symmetric around $I_p$. 
%\end{remark}

\begin{proof}[Proof of Corollary \ref{cor_lower_bound_SP_preliminary}]
We apply Proposition \ref{prop_van_Trees_abstract_version} with $G=SO(p)$ and $\psi=P_{\mathcal{I}}$. We start by checking that $P_{\II}:SO(p)\rightarrow \mathbb{R}^{p\times p}$, $P_{\II}(U)=\sum_{i\in\II}(Ue_i)(Ue_i)^T$ satisfies the boundedness assumption in Proposition \ref{prop_van_Trees_abstract_version}. By definition, for $U\in SO(p)$ and $\xi\in \mathfrak{so}(p)$, we have $dP_{\II}(U)U\xi=f'(0)$ with $f:\mathbb{R}\rightarrow \mathbb{R}^{p\times p},t\mapsto \sum_{i\in \II}(U\exp(t\xi)e_i)(U\exp(t\xi)e_i)^T$. Hence,
\begin{align}\label{eq_derivative_SP}
dP_{\II}(U)U\xi=U\Big(\xi\sum_{i\in \II}e_ie_i^{T}-\sum_{i\in \II}e_ie_i^{T}\xi\Big)U^T
\end{align}
and, using that the operator norm of an orthogonal matrix and an orthogonal projection is bounded by $1$, we get
\begin{align*}
\sup_{U\in SO(p)}\|dP_{\II}(U)U\xi\|_{\operatorname{HS}}\leq 2\|\xi\|_{\operatorname{HS}}.
\end{align*}
Hence, the assumptions of Proposition \ref{prop_van_Trees_abstract_version} are satisfied for the choices
\begin{align*}
 \xi^{(jk)}=c_{jk}L^{(jk)}\quad\text{and}\quad
 \xi^{(kj)}=c_{jk}L^{(jk)}\quad \text{for }j\in \II,k\notin \II
\end{align*}
and $\xi^{(jk)}=0$ in all other cases, with real numbers $c_{jk}$ to be chosen later. Then it follows from \eqref{eq_derivative_SP} that 
\begin{align*}
 d(P_{\II})_{jk}(U)U\xi^{(jk)}&=-c_{jk}e_j^TU(e_ke_j^T+e_je_k^T)U^Te_k=-c_{jk}(U_{jj}U_{kk}+U_{jk}U_{kj}),\\
 d(P_{\II})_{kj}(U)U\xi^{(kj)}&=-c_{jk}e_k^TU(e_ke_j^T+e_je_k^T)U^Te_j=-c_{jk}(U_{jj}U_{kk}+U_{jk}U_{kj}),
\end{align*}
for $j\in \II,k\notin \II$. Moreover, we have $\operatorname{tr}^2(U\xi^{(kj)})=\operatorname{tr}^2(U\xi^{(jk)})=c_{jk}^2(U_{jk}-U_{kj})^2$. Hence, Proposition~\ref{prop_van_Trees_abstract_version} and the fact that $\mathcal{I}_{I_p}$ is a bilinear form yield
\begin{align*}
&\int_{SO(p)} \mathbb{E}_{U}\|\hat P-P_{\II}(U)\|_{\operatorname{HS}}^2\,\tilde\pi(\operatorname{tr} (U))dU\\
&\geq 
\frac{2\Big(\sum_{j\in \II}\sum_{k\notin \II}c_{jk}\int_{SO(p)}(U_{jj}U_{kk}+U_{jk}U_{kj})\tilde\pi(\operatorname{tr} (U))\,dU\Big)^2}{\sum_{j\in \II}\sum_{k\notin\II}c_{jk}^2\Big(\mathcal{I}_{I_p}(L^{(jk)},L^{(jk)})+\int_{SO(p)}(U_{jk}-U_{kj})^2\frac{\tilde\pi'(\operatorname{tr}(U))^2}{\tilde\pi(\operatorname{tr}(U))}\,dU\Big)}\\
&=
\frac{2\Big(\sum_{j\in \II}\sum_{k\notin\II}c_{jk}\Big)^2\Big(\int_{SO(p)}(U_{11}U_{22}+U_{12}U_{21})\tilde\pi(\operatorname{tr} (U))\,dU\Big)^2}{\sum_{j\in \II}\sum_{k\notin\II}c_{jk}^2\Big(\mathcal{I}_{I_p}(L^{(jk)},L^{(jk)})+\int_{SO(p)} (U_{12}-U_{21})^2\frac{(\tilde\pi'(\operatorname{tr} (U)))^2}{\tilde\pi(\operatorname{tr} (U))}\,dU\Big)},
\end{align*} 
where the equality follows from Lemma \ref{LemHaarProp} below. The lower bound now follows from optimizing  in the $c_{jk}$, leading to the choices
\[
c_{jk}=\Big(\mathcal{I}_{I_p}(L^{(jk)},L^{(jk)})+\int_{SO(p)} (U_{12}-U_{21})^2\frac{(\tilde\pi'(\operatorname{tr} (U)))^2}{\tilde\pi(\operatorname{tr} (U))}\,dU\Big)^{-1}.
\]
In special case that $\tilde\pi=\tilde\pi_h$ is given by $\tilde\pi_h(s)=\exp(hps)/Z_h$ with $h>0$ and normalizing constant $Z_{h}$, we have the identity $(\tilde\pi_h'(\operatorname{tr} (U)))^2/\tilde\pi_h(\operatorname{tr} (U))=(hp)^2\tilde\pi_h(\operatorname{tr} (U))$. Hence, in this case, it holds that
\begin{align*}
\int_{SO(p)} (U_{12}-U_{21})^2\frac{(\tilde\pi_h'(\operatorname{tr} (U)))^2}{\tilde\pi_h(\operatorname{tr} (U))}\,dU=(hp)^2\int_{SO(p)} (U_{12}-U_{21})^2\tilde\pi_h(\operatorname{tr}(U))\,dU.
\end{align*}
Moreover, we have
\begin{align*}
\int_{SO(p)} &(U_{12}-U_{21})^2\tilde\pi_h(\operatorname{tr}(U))\,dU\leq 2\int_{SO(p)} (U_{12}^2+U_{21}^2)\tilde\pi_h(\operatorname{tr}(U))\,dU\\
&=\frac{2}{p-1}\int_{SO(p)} \sum_{j=2}^p(U_{1j}^2+U_{j1}^2)\tilde\pi_h(\operatorname{tr}(U))\,dU\leq \frac{4}{p-1}\leq \frac{8}{p},
\end{align*}
where the equality follows from Lemma \ref{LemHaarProp} below.
\end{proof}

\begin{lemma}\label{LemHaarProp}
The maps $i\mapsto \int U_{ii}\tilde\pi(\operatorname{tr}(U))\,dU$, $(i,j)\mapsto \int U_{ii}U_{jj}\tilde\pi(\operatorname{tr}(U))\,dU$, $(i,j)\mapsto \int U_{ij}U_{ji}\tilde\pi(\operatorname{tr}(U))\,dU$, and $(i,j)\mapsto \int(U_{ij}-U_{ji})^2\tilde\pi(\operatorname{tr}(U))\,dU$ are constant in $i,j=1,\dots,p$.  
\end{lemma}

\begin{proof}
For a permutation $\sigma$ on $\{1,\dots,p\}$, let $P_\sigma=(e_{\sigma(1)},\dots,e_{\sigma(p)})^T\in O(p)$ be the associated permutation matrix. Then multiplication from the left with $P_\sigma$ permutes the rows by $\sigma$ and multiplication from the right permutes the columns by $\sigma^{-1}$. By the properties of the Haar measure on the orthogonal group applied to the transformation $U\mapsto P_{(1i)}UP_{(1i)}$, which sends $U_{ii}$ to $U_{11}$ and leaves the trace fixed, the first claim follows. The other claims follow similarly from the transformation $U\mapsto P_{(2j)}P_{(1i)}UP_{(1i)}P_{(2j)}$, $i\neq j$.
\end{proof}

%\subsection{The simple case of $SO(2)$}
%Let us apply Corollary \ref{CorLowBoundPre} in the simple case $q=2$. Let $i\in I$ and $j\notin I$ and choose $J=\{i,j\}$. A matrix $U\in SO(2)$ has the form 
%\[
%U=\begin{pmatrix}
%\cos(t) & -\sin(t) \\
%\sin(t) & \cos(t)
%\end{pmatrix}
%\]  
%with $t\in[0,2\pi)$ and in this parametrization the Haar measure is given by $dU=dt/(2\pi)$. Consider, for instance the density function $(\operatorname{tr}U)^2/2=2\cos^2(t)$. Then we have
%\[
%\int_{SO(2)}(U_{11}U_{22}+U_{12}U_{21})\pi(\operatorname{tr} U)\,dU=\frac{1}{4\pi}\int_0^{2\pi}(\cos^2(t)-\sin^2(t))\cos^2(t)\,dt=1/2
%\]
%and 
%\[
%\int_{SO(p)} (U_{12}-U_{21})^2\frac{(\pi'(\operatorname{tr} U))^2}{\pi(\operatorname{tr} U)}\,dU=\frac{2}{\pi}\int_0^{2\pi}\sin^2(t)\,dt=2.
%\]
%Thus we obtain 
%\begin{align*}
%&\int_{SO(2)} R(\hat\psi,\psi(U))(\operatorname{tr}U)^2/2\,dU\geq \Big(\frac{n(\lambda_i-\lambda_j)^2}{\lambda_i\lambda_j}+2\Big)^{-1}\geq \frac{\lambda_i\lambda_j}{2n(\lambda_i-\lambda_j)^2}\wedge \frac{1}{4}.
%\end{align*}
%
\subsection{Designing optimal prior densities}\label{sec_design_optimal_densities}
In this section, we analyze the density function on $SO(p)$ defined by
\begin{align}\label{EqDefPriorDensity}
\pi_h(U)=\tilde\pi_h(\operatorname{tr}(U))=\frac{\exp(hp\operatorname{tr}(U))}{\int_{SO(p)}\exp(hp\operatorname{tr}(U))\,dU},\qquad U\in SO(p),
\end{align}
where $h>0$ is a real parameter.
It is clear that the resulting probability measure on $SO(p)$ approximates the Dirac measure at the identity matrix $I_p$ as $h\rightarrow \infty$. The following proposition gives a quantitative statement needed to deduce Theorem \ref{thm_lower_bound_Bayes_risk} from the second part of Corollary \ref{cor_lower_bound_SP_preliminary}.
\begin{proposition}\label{prop_exp_trace_density}
For each $\delta\in(0,1)$, there is a real number $h_\delta>0$ depending only on $\delta$ such that the density function given in \eqref{EqDefPriorDensity}  satisfies
\begin{align}\label{EqLDBound}
\int_{SO(p)}\operatorname{tr}(U)\tilde\pi_{h}(\operatorname{tr}(U))\,dU\geq (1-\delta)p\qquad\forall p\geq 2,h\geq h_\delta.
\end{align}
\end{proposition}
%
%\begin{remark}\label{RemIncreasingH} 
%The inequality \eqref{EqLDBound} also holds if we replace $h$ by $h'\geq h$. Indeed, the function $h'\mapsto \int_{SO(p)}\operatorname{tr}(U)\tilde\pi_{h'}(\operatorname{tr}(U))\,dU$ is non-decreasing in $h'$, as can be seen from the fact that its derivative is equal to
%\begin{align*}
% p\int_{SO(p)}(\operatorname{tr}(U))^2\tilde\pi_{h'}(\operatorname{tr}(U))\,dU-p\Big(\int_{SO(p)}\operatorname{tr}(U)\tilde\pi_{h'}(\operatorname{tr}(U))\,dU\Big)^2\geq 0.
%\end{align*}
%\end{remark}
%
While Proposition \ref{prop_exp_trace_density} can be deduced from Varadhan's lemma using that the Jacobi ensemble satisfies a large deviation principle, we present a more direct proof based on Weyl's integration formula in combination with some elementary large deviations arguments from \cite{MR2289756} (see also \cite{MR1746976} or \cite{MR2760897}). After that we explain how the claim can be alternatively obtained by more general large deviations arguments.
\begin{proof}[Proof of Proposition \ref{prop_exp_trace_density}]
Let us start with proving a lower bound for the exponential moment $\int_{SO(p)}\exp\big(hq\operatorname{tr} (U)\big)\,dU$.
We first consider the case $p=2m$ even. Then Weyl's integration formula (see, e.g., Theorem 3.5 in \cite{MR3971582}) in combination with the change of variables $y_i=\cos(\theta_i)$ gives
%\begin{align*}
%&\int_{SO(p)}\exp\big(hq\operatorname{tr} U\big)\,dU\nonumber\\
%&=\frac{2^{m^2-m+1}}{m!(2\pi)^m}\int_{[0,\pi]^m}e^{4hm\sum_{j=1}^m\cos(\theta_j)}\prod_{1\leq j<k\leq m}
%(\cos(\theta_k)-\cos(\theta_j))^2\,d\theta_1\dots d\theta_m
%\end{align*}
\begin{align*}
&\int_{SO(p)}\exp\big(hp\operatorname{tr} (U)\big)\,dU\nonumber\\
&=\frac{2^{m^2-m+1}}{m!(2\pi)^m}\int_{[-1,1]^m}e^{4hm\sum_{i=1}^my_i}\prod_{1\leq i<j\leq m}(y_j-y_i)^2\prod_{i=1}^m(1-y_i^2)^{-1/2}\prod_{i=1}^mdy_i.
\end{align*}
Now, let $0<\delta<1$ be a real number to be chosen later. Restricting the integral to the set
\begin{align*}
\Delta_m:=\Big\{y\in[-1,1]^m:1-\delta+\frac{\delta i}{m+1}\leq y_i\leq 1-\delta+\frac{\delta(i+1/2)}{m+1},1\leq i\leq m\Big\},
\end{align*}
we obtain
\begin{align}
&\int_{SO(p)}\exp\big(hp\operatorname{tr} (U)\big)\,dU\nonumber\\
&\geq\frac{2^{m^2-m+1}}{m!(2\pi)^m}\int_{\Delta_m}e^{4hm\sum_{i=1}^my_i}\prod_{1\leq i<j\leq m}(y_j-y_i)^2\prod_{i=1}^mdy_i\nonumber\\
&\geq\frac{2^{m^2-m+1}}{m!(2\pi)^m}\Big(\frac{\delta/2}{m+1}\Big)^me^{4hm^2(1-\delta/2)}\delta^{m^2-m}\prod_{1\leq i<j\leq m}\Big(\frac{j-i-1/2}{m+1}\Big)^2.\label{eq_LDP_technique}
\end{align}
Considering the last product term, we have the following bound
\begin{align*}
\prod_{1\leq i<j\leq m}\frac{j-i-1/2}{m+1}&=\prod_{k=1}^{m-1}\Big(\frac{k-1/2}{m+1}\Big)^{m-k}\\
&\geq \prod_{k=1}^{m}\Big(\frac{k}{m+1}\Big)^{m+1-k}\prod_{k=1}^{m}\Big(\frac{k-1/2}{k}\Big)^{m-k}.
\end{align*}
On the one hand, using that the function $t\mapsto (1-t)\log t$ is monotone increasing in $t\in(0,1]$ and equal to zero for $t=1$, we have
\begin{align*}
\frac{1}{(m+1)^2}\sum_{k=1}^m (m+1-k)\log\Big(\frac{k}{m+1}\Big)\geq \int_{0}^1(1-t)\log t\,dt =\frac{3}{4},
\end{align*}
and thus
\begin{align*}
\prod_{k=1}^{m}\Big(\frac{k}{m+1}\Big)^{m+1-k}\geq e^{-\frac{3(m+1)^2}{4}}.
\end{align*}
On the other hand, using the concavity of the function $t\mapsto \log t$, we have
\begin{align*}
\sum_{k=1}^{m} (m-k)(\log(k)-\log(k-1/2))\leq \sum_{k=1}^{m} \frac{m-k}{2k-1}\leq \sum_{k=1}^{m} \frac{m}{k}\leq m\log(em),
\end{align*}
and thus 
\begin{align*}
\prod_{k=1}^{m}\Big(\frac{k-1/2}{k}\Big)^{m-k}\geq e^{-m\log(em)}.
\end{align*}
Inserting these lower bounds into \eqref{eq_LDP_technique}, we get that the leading exponent with respect to $m$ is $m^2$. This allows us to conclude that there is a real number $h_\delta>0$ such that for every $h\geq h_\delta$ and every $p\geq 2$ even,
%
%Considering the product term, we have the following (crude) bound
%\begin{align*}
%&\prod_{1\leq j<k\leq m}\frac{k-j-1/2}{m+1}=\prod_{l=1}^{m-1}\Big(\frac{l-1/2}{m+1}\Big)^{m-l}\geq \Big(\frac{\Gamma(m-1/2)}{\Gamma(1/2)(m+1)^{m-1}}\Big)^{m-1}\\
%&\geq \Big(\frac{(m-1/2)^{m-3/2}}{\Gamma(1/2)e^{m-3/2}(m+1)^{m-1}}\Big)^{m-1}\geq \Big(\frac{1}{(2e)^m\sqrt{m}}\Big)^{m-1},
%\end{align*}
%where the last inequality can be seen from the inequality $\Gamma(x)\geq (x/e)^{x-1}$, $x\geq 1$. 
\begin{align}\label{EqLDBound1}
\int_{SO(p)}\exp(h p\operatorname{tr} (U))\,dU\geq \exp(h p^2(1-\delta)).
\end{align}
The same bound also holds for $p$ odd, in which case Weyl's integration formula has a slightly different form (see, e.g., Theorem 3.5 in \cite{MR3971582}). Next, using that the logarithmic moment-generating function 
\begin{align*}
\psi(h)=\log\Big(\int_{SO(p)}\exp(h p\operatorname{tr} (U))\,dU\Big),\qquad h\in\mathbb{R},
\end{align*}
is convex, we get
\begin{align*}
\int_{SO(p)}p\operatorname{tr}(U)\tilde\pi_{h}(\operatorname{tr}(U))\,dU=\psi'(h)\geq h^{-1}\psi(h)\geq p^2(1-\delta)
\end{align*}
for every $h\geq h_\delta$ and every $p\geq 2$, where we used convexity and $\psi(0)=0$ in the first inequality and \eqref{EqLDBound1} in the last inequality. Dividing both sides by $p$, the claim follows.
\end{proof} 

The following corollary deals with the prefactor from the second lower bound in Corollary \ref{cor_lower_bound_SP_preliminary}.

\begin{corollary}\label{cor_LD_result}
For each $\delta\in(0,1)$ there is a real number $h_\delta>0$ such that the density function given in \eqref{EqDefPriorDensity} satisfies, for every $p\geq 2$ and every $h\geq h_\delta$,
\begin{align*}
\int_{SO(p)}(U_{11}U_{22}+U_{12}U_{21})\tilde\pi_h(\operatorname{tr}(U))\,dU&\geq (1-\delta)^2-\frac{2(1-(1-\delta)^2)}{p-1}.
\end{align*}
\end{corollary}
\begin{proof}
Consider $h_\delta$ from Proposition \ref{prop_exp_trace_density} and let $h\geq h_\delta$. By the Cauchy-Schwarz inequality and Proposition \ref{prop_exp_trace_density}, we have
\[
\int_{SO(p)}(\operatorname{tr}(U))^2\tilde\pi_h(\operatorname{tr}(U))\,dU\geq \Big(\int_{SO(p)}\operatorname{tr}(U)\tilde\pi_h(\operatorname{tr}(U))\,dU\Big)^2\geq (1-\delta)^2p^2.
\]
On the other hand, by Lemma \ref{LemHaarProp}, we have
\begin{align*}
&\int_{SO(p)}(\operatorname{tr}(U))^2\tilde\pi_h(\operatorname{tr}(U))\,dU\\
&=\sum_{i=1}^p\int_{SO(p)}U_{ii}^2\tilde\pi_h(\operatorname{tr}(U))\,dU+\sum_{i\neq j}\int_{SO(p)}U_{ii}U_{jj}\tilde\pi_h(\operatorname{tr}(U))\,dU\\
&=p\int_{SO(p)}U_{11}^2\tilde\pi_h(\operatorname{tr}(U))\,dU+p(p-1)\int_{SO(p)}U_{11}U_{22}\tilde\pi_h(\operatorname{tr}(U))\,dU.
\end{align*}
Now, $\int U_{11}^2\tilde\pi_h(\operatorname{tr}(U))\,dU\leq 1$ and thus
\begin{align}\label{proof_cor_LD_result_part1}
\int_{SO(p)}U_{11}U_{22}\tilde\pi_h(\operatorname{tr}(U))\,dU\geq \frac{(1-\delta)^2p-1}{p-1}=(1-\delta)^2-\frac{1-(1-\delta)^2}{p-1}.
\end{align}
This gives the first part of the integral. For the second part, Lemma \ref{LemHaarProp} gives
\[
\int_{SO(p)}U_{12}U_{21}\tilde\pi_h(\operatorname{tr}(U))\,dU=\frac{1}{p-1}\sum_{k=2}^p\int_{SO(p)}U_{1k}U_{k1}\tilde\pi_h(\operatorname{tr}(U))\,dU.
\]
Hence
\begin{align}
\Big|\int_{SO(p)}U_{12}U_{21}\tilde\pi_h(\operatorname{tr}(U))\,dU\Big|&\leq \frac{1}{p-1}\sum_{k=2}^p\int_{SO(p)}|U_{1k}U_{k1}|\tilde\pi_h(\operatorname{tr}(U))\,dU \nonumber\\
&\leq \frac{1}{2(p-1)}\sum_{k=2}^p\int_{SO(p)}(U_{1k}^2+U_{k1}^2)\tilde\pi_h(\operatorname{tr}(U))\,dU \nonumber\\
&=\frac{1}{p-1}\int_{SO(p)}(1-U_{11}^2)\tilde\pi_h(\operatorname{tr}(U))\,dU.\label{eq_mixed_term}
\end{align}
By the Cauchy-Schwarz inequality, Lemma \ref{LemHaarProp} and Proposition \ref{prop_exp_trace_density}, we have 
\begin{align*}
\int_{SO(p)}U_{11}^2\tilde\pi_h(\operatorname{tr}(U))\,dU\geq \Big(\int_{SO(p)}U_{11}\tilde\pi_h(\operatorname{tr}(U))\,dU\Big)^2\geq (1-\delta)^2,
\end{align*}
and inserting this into \eqref{eq_mixed_term}, we get
\begin{align}\label{proof_cor_LD_result_part2}
\Big|\int_{SO(p)}U_{12}U_{21}\tilde\pi_h(\operatorname{tr}(U))\,dU\Big|&\leq \frac{1-(1-\delta)^2}{p-1}.
\end{align}
The claim now follows from combining \eqref{proof_cor_LD_result_part1} and \eqref{proof_cor_LD_result_part2}.
\end{proof}
We conclude this section by outlining a slightly different proof of Proposition \ref{prop_exp_trace_density}. Making the change of variables $x_j=(1+y_j)/2$ at the beginning of the above proof leads to
\begin{align}
&\int_{SO(p)}\exp\big(hp\operatorname{tr} (U)\big)\,dU=\int_{[0,1]^m}e^{8hm\sum_{i=1}^mx_i-4hm^2}d\nu_m(x_1,\dots,x_m)\label{EqWeylInt}
\end{align}
with probability measure $\nu_m$ on $[0,1]^m$ defined by
\begin{align*}
 d\nu_m(x_1,\dots,x_m)=\frac{2^{2m^2-2m+1}}{m!(2\pi)^m}\prod_{i=1}^m(x_i(1-x_i))^{-1/2}\prod_{1\leq i<j\leq m}(x_j-x_i)^2\prod_{i=1}^mdx_i.
\end{align*}
This measure corresponds to a Jacobi ensemble and we next state a large deviation principle proved in \cite{MR2289756}. For this let $\mathcal{M}_1([0,1])$ be the space of all probability measures on $[0,1]$ endowed with the usual weak topology (metrizable by the bounded Lipschitz metric), and let $\Sigma:\mathcal{M}_1([0,1])\rightarrow [-\infty,-2\log 2]$ be the noncommutative entropy define by 
\begin{align*}
\Sigma(\mu)=\int_0^1\int_0^1\log|x-y|\,d\mu(x)d\mu(y),\qquad\mu\in\mathcal{M}_1([0,1]).
\end{align*}
Assume now that the random vector $x^{(m)}$ is distributed according to the probability measure $\nu_m$ on $[0,1]^m$. Then it follows from Proposition 2.1 in \cite{MR2289756} (resp.~a slight variation of it with $\kappa(N)$ and $\lambda(N)$ allowed to be negative such that $\kappa(N)/N,\lambda(N)/N\rightarrow 0$) that the empirical measure 
\[
\frac{1}{m}\delta_{x_1^{(m)}}+\dots+\frac{1}{m}\delta_{x_m^{(m)}}
\] satisfies the large deviation principle with speed $m^2$ and (good) rate function $\mathcal{I}(\mu)=-\Sigma(\mu)-2\log 2$, $\mu\in \mathcal{M}_1([0,1])$. We now apply Varadhan's integral lemma to the function $f(\mu)=8h\int_0^1y\,d\mu(y)$. This function is indeed continuous and bounded and it follows from \cite[Theorem 27.10(i)]{MR1876169} or \cite[Theorem 4.3.1]{MR2571413} that for every $h>0$,
\begin{align}\label{eq_question}
\lim_{m\rightarrow\infty}\frac{1}{m^2}\log\mathbb{E}\exp\Big(8hm\sum_{i=1}^mx_i^{(m)}\Big)=\sup_{\mu\in \mathcal{M}_1([0,1])}\Big(8h\int_0^1y\,d\mu(y)-\mathcal{I}(\mu)\Big).
\end{align}
From this point we can argue as above. Note that the above proof corresponds to choosing $\mu$ as the uniform measure on $[1-\delta/2,1]$.
Note that it is possible to show that the right-hand side of \eqref{eq_question} is equal to $4h+2h^2$ if $2h\leq 1$ and equal to $4h-\log(2h)+4h-3/2$ if $2h>1$. This can be achieved by relating the right-hand side of \eqref{eq_question} to the rate function in a large deviation result for the Gross-Witten-Wadia model \cite[p. 82]{MR1743092} and \cite[Proposition 5.3.10]{MR1746976}. We omit here the details. Translating this back to the left-hand side of \eqref{EqWeylInt}, we get
\begin{align*}
 \lim_{p\rightarrow \infty}\frac{1}{p^2}\log\Big(\int_{SO(p)}\exp\big(hp\operatorname{tr} (U)\big)\,dU\Big)\rightarrow \begin{cases}\frac{h^2}{2},&\quad 2h\leq 1,\\
 h-\frac{1}{4}\log(2h)-\frac{3}{8},&\quad 2h\geq 1.
 \end{cases}
\end{align*} 
In this paper the main focus is on nonasymptotic lower bounds in which case the latter exact limit has no special meaning. Yet, let us mention that in the asymptotic scenario $n,p\rightarrow\infty$ with $p/n\rightarrow \gamma\in(0,\infty)$, it allows to obtain more precise constants than presented in Corollary \ref{cor_LD_result} (although the result is not enough to exactly compute the asymptotic value of the integral in Corollary~\ref{cor_LD_result}). All this leads to the question to prove an (asymptotic) version of Corollary \ref{cor_scm} with constants as sharp as possible.

\subsection{Application: principal component analysis}\label{proof_thm_lower_bound_Bayes_risk}
\begin{proof}[End of proof of Theorem \ref{thm_lower_bound_Bayes_risk}]
Consider the statistical model $\{\mathbb{P}_U:U\in SO(p)\}$ from \eqref{eq_stat_experiment_version} with $\mathbb{P}_U=\mathcal{N}(0,U\Lambda U^T)^{\otimes n}$, $\Lambda=\operatorname{diag}(\lambda_1,\dots,\lambda_p)$ and $\lambda_1\geq \dots\geq \lambda_p> 0$. By Example \ref{ex_PCA_equivariant} and Lemma \ref{LemChi2Fisher}, we know that Assumptions \ref{ass_equivariant} and~\ref{ass_chi_square} are satisfied with $\mathcal{I}_{I_p}(L^{(ij)},L^{(ij)})=n(\lambda_i-\lambda_j)^2/\lambda_i\lambda_j$. Hence, combining the second claim in Corollary \ref{cor_lower_bound_SP_preliminary} with Corollary \ref{cor_LD_result}, we conclude that for every $\delta\in(0,1)$, there is a constant $C_\delta>0$ depending only on $\delta$ such that 
\begin{align*}
&\inf_{\hat P}\int_{SO(p)}\mathbb{E}_{U}\|\hat P-P_{\II}(U)\|_{\operatorname{HS}}^2\,\tilde\pi_h(\operatorname{tr}(U))dU\\
&\geq  2(1-\delta)\sum_{i\in \II}\sum_{j\notin\II}\Big(\frac{n(\lambda_i-\lambda_j)^2}{\lambda_i\lambda_j}+8h^2p\Big)^{-1}\qquad\forall h\geq C_\delta.
\end{align*} 
Inserting the inequality 
\begin{align}\label{eq:transfer:-:min}
(x+y)^{-1}\geq (1-\delta)x^{-1}\wedge \delta y^{-1},\qquad x,y\geq 0,
\end{align}
the claim follows
\end{proof}

\subsection{Application: matrix denoising}
Consider the statistical model $\{\mathbb{P}_U:U\in SO(p)\}$ from \eqref{eq_stat_experiment_low_rank} with $\mathbb{P}_U$ being the distribution of $X=U\Lambda U^T+\sigma W$, where $\sigma>0$, $W$ is a GOE matrix, and $\Lambda=\operatorname{diag}(\lambda_1,\dots,\lambda_p)$ with $\lambda_1\geq \dots\geq \lambda_p\geq 0$. By Example \ref{ex_matrix_denoising_equivariant} and Lemma \ref{LemChi2Fisher_low_rank}, we know that Assumptions \ref{ass_equivariant} and \ref{ass_chi_square} are satisfied with $\mathcal{I}_{I_p}(L^{(ij)},L^{(ij)})=(\lambda_i-\lambda_j)^2/\sigma^2$. Hence, combining the second claim in Corollary \ref{cor_lower_bound_SP_preliminary} with Corollary \ref{cor_LD_result}, we conclude that for every $\delta\in(0,1)$, there is a constant $C_\delta>0$ depending only on $\delta$ such that  
\begin{align*}
&\inf_{\hat P}\int_{SO(p)}\mathbb{E}_{U}\|\hat P-P_{\II}(U)\|_{\operatorname{HS}}^2\,\tilde\pi_h(\operatorname{tr}(U))dU\\
&\geq  2(1-\delta)\sum_{i\in \II}\sum_{j\notin\II}\Big(\frac{(\lambda_i-\lambda_j)^2}{\sigma^2}+8h^2p\Big)^{-1}\qquad\forall h\geq C_\delta.
\end{align*} 
Applying \eqref{eq:transfer:-:min}, we get that
\begin{align*}
&\inf_{\hat P}\int_{SO(p)}\mathbb{E}_{U}\|\hat P-P_{\II}(U)\|_{\operatorname{HS}}^2\,\tilde\pi_h(\operatorname{tr}(U))dU\nonumber\\
&\geq  2(1-\delta)^2\sum_{i\in \II}\sum_{j\notin\II}\min\Big(\frac{\sigma^2}{(\lambda_i-\lambda_j)^2},\frac{\delta}{8h^2p}\Big)\qquad\forall h\geq C_\delta.
\end{align*}

%%%%%%%%%%%%%%%%%%%%%%%%%%%%%%%%%%%%%%%%%%%

\appendix
\section{Additional proofs}\label{sec_appendix}
\begin{proof}[Proof of Corollary \ref{cor_polynomial}]
By convexity, we have $(j^\alpha-i^{\alpha})/(j-i)\le \alpha j^{\alpha-1}$ for every $j>i$ and thus
\begin{align*}
\frac{\lambda_i\lambda_j}{(\lambda_i-\lambda_j)^2}=\frac{i^\alpha j^\alpha}{(j^\alpha-i^\alpha)^2}\geq \alpha^{-2}\frac{i^\alpha j^{2-\alpha}}{(j-i)^2}\qquad\forall j>i.
\end{align*}
We now assume for simplicity that $d$ is even and choose $\JJ=\{d/2+1,d/2+2,\dots, d+d/2\}$ such that $|\JJ|=d$. Then 
\begin{align*}
&\sum_{i\in \II\cap \JJ}\sum_{j\in \JJ\setminus \II}\Big(\frac{\lambda_i\lambda_j}{n(\lambda_i-\lambda_j)^2}\wedge \frac{1}{|\JJ|}\Big)=\sum_{d/2<i\leq d}\sum_{d<j\leq 3d/2}\Big(\frac{\lambda_i\lambda_j}{n(\lambda_i-\lambda_j)^2}\wedge \frac{1}{d}\Big)\\
&\geq \sum_{d/2<i\leq d}\sum_{d<j\leq 3d/2}\Big(\frac{\alpha^{-2}3^{-\alpha}d^2}{n(j-i)^2}\wedge \frac{1}{d}\Big)\geq \sum_{k=1}^{d/2}\Big(\frac{\alpha^{-2}3^{-\alpha}d^2}{nk}\wedge \frac{k}{d}\Big),
\end{align*}
where we used in the last inequality that for $k\leq d/2$ the number of indices $(i,j)$ in the double sum satisfying $j-i=k$ is equal to $k$. 
We conclude that 
\begin{align*}
\sum_{i\in \II\cap \JJ}\sum_{j\in \JJ\setminus \II}\Big(\frac{\lambda_i\lambda_j}{n(\lambda_i-\lambda_j)^2}\wedge \frac{1}{|\JJ|}\Big)\geq c\sum_{k=1}^{d/2}\Big(\frac{(d/2)^2}{nk}\wedge \frac{k}{(d/2)}\Big)
\end{align*}
with $c=\alpha^{-2}3^{-\alpha}2^{-1}$. Hence Corollary \ref{cor_polynomial} follows from Theorem \ref{thm_minimax_Hilbert} and the following lemma applied with $x=(d/2)^2/n$ and $m=d/2$.
\begin{lemma}\label{lem_elementary}
For every $m\geq 1$ and every $x\geq 0$, we have 
\begin{align*}
\sum_{k=1}^{m}\Big(\frac{x}{k}\wedge \frac{k}{m}\Big)\geq c\cdot\min\Big(m,x+x\log_+\Big(m\wedge\sqrt{\frac{m}{x}}\Big)\Big)
\end{align*}
for some absolute constant $c>0$.
\end{lemma}
It remains to prove Lemma \ref{lem_elementary}. If $x\geq m$, then we have 
\begin{align}\label{eq_large_x}
(I):=\sum_{k=1}^{m}\Big(\frac{x}{k}\wedge \frac{k}{m}\Big)=\sum_{k=1}^{m}\frac{k}{m}=\frac{m+1}{2}.
\end{align}
On the other hand if $x\leq 1/m$, then 
\begin{align}\label{eq_small_x}
(I)=\sum_{k=1}^{m}\frac{x}{k}\geq x\int_1^{m+1}\frac{1}{t}\,dt=x\log(m+1).
\end{align}
Finally, if $1/m<x<m$, then let $k_0\geq 2$ be the smallest natural number in $\{1,\dots,m\}$ such that $x/k_0\leq k_0/m$. Then we have
\begin{align*}
(I)=\sum_{k=1}^{k_0-1}\frac{k}{m}+\sum_{k=k_0}^{m}\frac{x}{k}\geq \frac{k_0(k_0-1)}{2m}+x\log\Big(\frac{m+1}{k_0}\Big),
\end{align*}
where the second sums is treated as in \eqref{eq_small_x}. Using the definition of $k_0$, we get
\begin{align*}
(I)&\geq x\frac{k_0-1}{2k_0}+x\log\Big(\frac{m+1}{\sqrt{xm}+1}\Big)\geq \frac{x}{4}+x\log\Big(\frac{m+1}{\sqrt{xm}+1}\Big).
\end{align*}
Hence, for $1/m<x<m$, we get 
\begin{align}\label{eq_middle_x}
(I)\geq \frac{x}{8}+\frac{x}{8}\log\Big(\frac{e(m+1)}{\sqrt{xm}+1}\Big)\geq \frac{x}{8}+\frac{x}{8}\log\Big(\sqrt{\frac{m}{x}}\Big).
\end{align}
Collecting \eqref{eq_large_x}--\eqref{eq_middle_x}, the claim follows.
\end{proof}

\begin{proof}[Proof of \eqref{eq_matrix_exp}]
The second and third power of 
\begin{align*}
\xi=\sum_{j\neq i}x_j(e_ie_j^T-e_je_i^T),\qquad  \sum_{j\neq i}x_j^2=1,
\end{align*}
are given by
\begin{align*}
\xi^2=-e_ie_i^T-\sum_{j\neq i}\sum_{k\neq i}x_jx_ke_je_k^T\quad\text{and}\quad \xi^3=-\xi,
\end{align*}
as can be seen from the fact that $e_1,\dots,e_p$ is an orthonormal basis and the relation $\sum_{j\neq i}x_j^2=1$. Hence, we have $\xi^{2n-1}=(-1)^{n-1}\xi$ and $\xi^{2n}=(-1)^{n-1}\xi^2$ for all $n\geq 1$. Combining these relations with $\exp(t\xi)=\sum_{n\geq 0}(t\xi)^n/n!$ and $\xi^0=I_p$, we get
\begin{align*}
\exp(t\xi)_{ii}&=e_i^T\exp(t\xi)e_i=1-\sum_{n\geq 1}(-1)^{n-1}\frac{t^{2n}}{(2n)!}=\cos (t),\\
\exp(t\xi)_{ji}&=e_j^T\exp(t\xi)e_i=-x_j\sum_{n\geq 1}(-1)^{n-1}\frac{t^{2n-1}}{(2n-1)!}=-x_j\sin (t),\quad j\neq i,
\end{align*}
which gives the claim.
\end{proof}

\begin{proof}[Proof of \eqref{eq_chi_square_comp_density}]
For the first claim, note that we have $\Pi\circ R_{-1}(1)=1-\Pi\circ R_{-1}(-1)=1-q$ and thus
\begin{align*}
\chi^2(\Pi\circ R_{-1},\Pi)=(1-q)\Big(\frac{q}{1-q}\Big)^2+q\Big(\frac{1-q}{q}\Big)^2-1=\frac{(1-2q)^2}{q(1-q)}.
\end{align*}
To see the second claim, let $P_{g}$ be the distribution of observation $X_1$ with density $f_g$. Then we have
\begin{align*}
 \chi^2(P_{-1},P_1)&=\int_{-1}^1(f_{-1}(x)-f_1(x))^2\frac{1}{f_1(x)}\,dx\\
 &\leq 4\int_{-1}^1(f_{-1}(x)-f_1(x))^2\,dx=2^5c_0^2h^{2\beta+1}\|K\|_{L^2}^2,
\end{align*}
where the inequality follows from the fact that $f_{1}(x)\geq 1/4$ for all $x\in[-1,1]$  by construction.
Hence, by \eqref{EqChiSquareP1}, we have
\begin{align*}
\chi^2(\mathbb{P}_{-1},\mathbb{P}_1)=(1+\chi^2(P_{-1},P_1))^n-1\leq e^{2^5c_0^2h^{2\beta+1}n\|K\|_{L^2}^2}-1,
\end{align*}
which gives the second claim.
\end{proof}

\subsection*{Acknowledgement} The author would like to thank Markus Rei\ss{}, Holger K\"osters and the two anonymous referees for their helpful comments and remarks and Alain Rouault for very helpful discussions and comments.

\bibliographystyle{plain}
\bibliography{lit}

\begin{thebibliography}{10}

\bibitem{MR2760897}
G.~W. Anderson, A.~Guionnet, and O.~Zeitouni.
\newblock {\em An introduction to random matrices}.
\newblock Cambridge University Press, Cambridge, 2010.

\bibitem{MR3701408}
N.~Ay, J.~Jost, H.~V. L\^{e}, and L.~Schwachh\"{o}fer.
\newblock {\em Information geometry}.
\newblock Springer, Cham, 2017.

\bibitem{MR1312157}
R.~G. Bartle.
\newblock {\em The elements of integration and {L}ebesgue measure}.
\newblock Wiley Classics Library. John Wiley \& Sons, Inc., New York, 1995.

\bibitem{DBLP:conf/colt/Belkin18}
M.~Belkin.
\newblock Approximation beats concentration? an approximation view on inference
  with smooth radial kernels.
\newblock In {\em Conference On Learning Theory, {COLT}}, pages 1348--1361,
  2018.

\bibitem{MR1477662}
R.~Bhatia.
\newblock {\em Matrix analysis}.
\newblock Springer-Verlag, New York, 1997.

\bibitem{MR3833647}
G.~Blanchard and N.~M\"{u}cke.
\newblock Optimal rates for regularization of statistical inverse learning
  problems.
\newblock {\em Found. Comput. Math.}, 18(4):971--1013, 2018.

\bibitem{CLM20}
T.~Cai, H.~Li, and R.~Ma.
\newblock Optimal structured principal subspace estimation: Metric entropy and
  minimax rates.
\newblock {\em Journal of Machine Learning Research}, 22(46), 2021.

\bibitem{MR3161458}
T.~Cai, Z.~Ma, and Y.~Wu.
\newblock Sparse {PCA}: optimal rates and adaptive estimation.
\newblock {\em Ann. Statist.}, 41(6):3074--3110, 2013.

\bibitem{MR3334281}
T.~Cai, Z.~Ma, and Y.~Wu.
\newblock Optimal estimation and rank detection for sparse spiked covariance
  matrices.
\newblock {\em Probab. Theory Related Fields}, 161(3-4):781--815, 2015.

\bibitem{MR3766946}
T.~T. Cai and A.~Zhang.
\newblock Rate-optimal perturbation bounds for singular subspaces with
  applications to high-dimensional statistics.
\newblock {\em Ann. Statist.}, 46(1):60--89, 2018.

\bibitem{MR3236753}
G.~Da~Prato and J.~Zabczyk.
\newblock {\em Stochastic equations in infinite dimensions}.
\newblock Cambridge University Press, Cambridge, second edition, 2014.

\bibitem{MR650934}
J.~Dauxois, A.~Pousse, and Y.~Romain.
\newblock Asymptotic theory for the principal component analysis of a vector
  random function: some applications to statistical inference.
\newblock {\em J. Multivariate Anal.}, 12(1):136--154, 1982.

\bibitem{MR2571413}
A.~Dembo and O.~Zeitouni.
\newblock {\em Large deviations techniques and applications}.
\newblock Springer-Verlag, Berlin, 2010.
\newblock Corrected reprint of the second (1998) edition.

\bibitem{MR1932358}
R.~M. Dudley.
\newblock {\em Real analysis and probability}.
\newblock Cambridge University Press, Cambridge, 2002.
\newblock Revised reprint of the 1989 original.

\bibitem{MR1089423}
M.~L. Eaton.
\newblock {\em Group invariance applications in statistics}, volume~1 of {\em
  NSF-CBMS Regional Conference Series in Probability and Statistics}.
\newblock Institute of Mathematical Statistics, Hayward, CA; American
  Statistical Association, Alexandria, VA, 1989.

\bibitem{MR2431769}
M.~L. Eaton.
\newblock {\em Multivariate statistics: A vector space approach}.
\newblock Institute of Mathematical Statistics, Beachwood, OH, 2007.
\newblock Reprint of the 1983 original.

\bibitem{MR770934}
R.~H. Farrell.
\newblock {\em Multivariate calculation}.
\newblock Springer Series in Statistics. Springer-Verlag, New York, 1985.

\bibitem{GaoZhang}
C.~Gao and A.~Y. Zhang.
\newblock Exact minimax estimation for phase synchronization.
\newblock Available at https://arxiv.org/abs/2010.04345.

\bibitem{MR1354456}
R.~D. Gill and Boris~Y. Levit.
\newblock Applications of the {V}an {T}rees inequality: a {B}ayesian
  {C}ram\'{e}r-{R}ao bound.
\newblock {\em Bernoulli}, 1(1-2):59--79, 1995.

\bibitem{MR0400513}
J.~H\'{a}jek.
\newblock Local asymptotic minimax and admissibility in estimation.
\newblock In {\em Proceedings of the {S}ixth {B}erkeley {S}ymposium on
  {M}athematical {S}tatistics and {P}robability, {V}ol. {I}: {T}heory of
  statistics}, pages 175--194, 1972.

\bibitem{MR2332269}
P.~Hall and J.~L. Horowitz.
\newblock Methodology and convergence rates for functional linear regression.
\newblock {\em Ann. Statist.}, 35(1):70--91, 2007.

\bibitem{MR1743092}
F.~Hiai and D.~Petz.
\newblock A large deviation theorem for the empirical eigenvalue distribution
  of random unitary matrices.
\newblock {\em Ann. Inst. H. Poincar\'{e} Probab. Statist.}, 36(1):71--85,
  2000.

\bibitem{MR1746976}
F.~Hiai and D.~Petz.
\newblock {\em The semicircle law, free random variables and entropy}.
\newblock American Mathematical Society, Providence, RI, 2000.

\bibitem{MR2289756}
F.~Hiai and D.~Petz.
\newblock Large deviations for functions of two random projection matrices.
\newblock {\em Acta Sci. Math. (Szeged)}, 72(3-4), 2006.

\bibitem{MR3379106}
T.~Hsing and R.~Eubank.
\newblock {\em Theoretical foundations of functional data analysis, with an
  introduction to linear operators}.
\newblock John Wiley \& Sons, Ltd., Chichester, 2015.

\bibitem{MR620321}
I.~A. Ibragimov and R.~Z. Has'minskii.
\newblock {\em Statistical estimation: Asymptotic theory}.
\newblock Springer-Verlag, New York-Berlin, 1981.

\bibitem{JW18}
M.~Jirak and M.~Wahl.
\newblock Relative perturbation bounds with applications to empirical
  covariance operators.
\newblock Available at https://arxiv.org/pdf/1802.02869, 2018.

\bibitem{MR4052188}
M.~Jirak and M.~Wahl.
\newblock Perturbation bounds for eigenspaces under a relative gap condition.
\newblock {\em Proc. Amer. Math. Soc.}, 148(2):479--494, 2020.

\bibitem{MR2334195}
I.~M. Johnstone.
\newblock High dimensional statistical inference and random matrices.
\newblock In {\em International {C}ongress of {M}athematicians. {V}ol. {I}},
  pages 307--333. Eur. Math. Soc., Z\"{u}rich, 2007.

\bibitem{MR2651957}
P.~E. Jupp.
\newblock A van {T}rees inequality for estimators on manifolds.
\newblock {\em J. Multivariate Anal.}, 101(8):1814--1825, 2010.

\bibitem{MR1876169}
O.~Kallenberg.
\newblock {\em Foundations of modern probability}.
\newblock Springer-Verlag, New York, second edition, 2002.

\bibitem{K17}
V.~Koltchinskii.
\newblock Asymptotically efficient estimation of smooth functionals of
  covariance operators.
\newblock {\em J. Eur. Math. Soc. (JEMS)}, 23(3):765--843, 2021.

\bibitem{MR4065170}
V.~Koltchinskii, M.~L\"{o}ffler, and R.~Nickl.
\newblock Efficient estimation of linear functionals of principal components.
\newblock {\em Ann. Statist.}, 48(1):464--490, 2020.

\bibitem{MR1698873}
J.~R. Magnus and H.~Neudecker.
\newblock {\em Matrix differential calculus with applications in statistics and
  econometrics}.
\newblock John Wiley \& Sons, Ltd., Chichester, 1999.
\newblock Revised reprint of the 1988 original.

\bibitem{MR3300524}
A.~Mas and F.~Ruymgaart.
\newblock High-dimensional principal projections.
\newblock {\em Complex Anal. Oper. Theory}, 9(1):35--63, 2015.

\bibitem{MR3971582}
E.~S. Meckes.
\newblock {\em The random matrix theory of the classical compact groups}.
\newblock Cambridge University Press, Cambridge, 2019.

\bibitem{MR652932}
R.~J. Muirhead.
\newblock {\em Aspects of multivariate statistical theory}.
\newblock John Wiley \& Sons, Inc., New York, 1982.

\bibitem{N08}
B.~Nadler.
\newblock Finite sample approximation results for principal component analysis:
  a matrix perturbation approach.
\newblock {\em Ann. Statist.}, 36:2791--2817, 2008.

\bibitem{MR3739989}
S.~O'Rourke, V.~Vu, and K.~Wang.
\newblock Random perturbation of low rank matrices: improving classical bounds.
\newblock {\em Linear Algebra Appl.}, 540:26--59, 2018.

\bibitem{MR1665590}
A.~Pajor.
\newblock Metric entropy of the {G}rassmann manifold.
\newblock In {\em Convex geometric analysis ({B}erkeley, {CA}, 1996)},
  volume~34 of {\em Math. Sci. Res. Inst. Publ.}, pages 181--188. Cambridge
  Univ. Press, Cambridge.

\bibitem{P07}
D.~Paul.
\newblock Asymptotics of sample eigenstructure for a large dimensional spiked
  covariance model.
\newblock {\em Statist. Sinica}, 17:1617--1642, 2007.

\bibitem{MR3845022}
A.~Perry, A.~S. Wein, A.~S. Bandeira, and A.~Moitra.
\newblock Optimality and sub-optimality of {PCA} {I}: {S}piked random matrix
  models.
\newblock {\em Ann. Statist.}, 46(5):2416--2451, 2018.

\bibitem{MR3862091}
Amelia Perry, Alexander~S. Wein, Afonso~S. Bandeira, and Ankur Moitra.
\newblock Message-passing algorithms for synchronization problems over compact
  groups.
\newblock {\em Comm. Pure Appl. Math.}, 71(11):2275--2322, 2018.

\bibitem{IMM2012-03274}
K.~B. Petersen and M.~S. Pedersen.
\newblock The matrix cookbook, nov 2012.
\newblock Version 20121115.

\bibitem{R18}
M.~Rei\ss{}.
\newblock On the reconstruction error of {PCA}.
\newblock {\em Oberwolfach Reports}, 29:1773--1774, 2018.

\bibitem{ReissWahl}
M.~Reiss and M.~Wahl.
\newblock Nonasymptotic upper bounds for the reconstruction error of {PCA}.
\newblock {\em Ann. Statist.}, 48(2):1098--1123, 2020.

\bibitem{R14}
L.~R\"{u}schendorf.
\newblock {\em Mathematische Statistik}.
\newblock Springer-Verlag Berlin Heidelberg, 2014.

\bibitem{MR2450103}
I.~Steinwart and A.~Christmann.
\newblock {\em Support vector machines}.
\newblock Information Science and Statistics. Springer, New York, 2008.

\bibitem{conf/colt/SteinwartHS09}
I.~Steinwart, D.~R. Hush, and C.~Scovel.
\newblock Optimal rates for regularized least squares regression.
\newblock In {\em COLT}, 2009.

\bibitem{MR812467}
H.~Strasser.
\newblock {\em Mathematical theory of statistics}.
\newblock Walter de Gruyter \& Co., Berlin, 1985.

\bibitem{MR2724359}
A.~B. Tsybakov.
\newblock {\em Introduction to nonparametric estimation}.
\newblock Springer, New York, 2009.
\newblock Revised and extended from the 2004 French original.

\bibitem{MR1652247}
A.~W. van~der Vaart.
\newblock {\em Asymptotic statistics}.
\newblock Cambridge University Press, Cambridge, 1998.

\bibitem{MR3161452}
V.~Q. Vu and J.~Lei.
\newblock Minimax sparse principal subspace estimation in high dimensions.
\newblock {\em Ann. Statist.}, 41(6):2905--2947, 2013.

\bibitem{Wahl}
M.~Wahl.
\newblock On the perturbation series for eigenvalues and eigenprojections.
\newblock Available at https://arxiv.org/abs/1910.08460, 2019.

\bibitem{MR943833}
H.~Witting.
\newblock {\em Mathematische {S}tatistik. {I}}.
\newblock B. G. Teubner, Stuttgart, 1985.

\bibitem{MR1311972}
R.~Yang and J.~O. Berger.
\newblock Estimation of a covariance matrix using the reference prior.
\newblock {\em Ann. Statist.}, 22(3):1195--1211, 1994.

\bibitem{MR1462963}
B.~Yu.
\newblock Assouad, {F}ano, and {L}e {C}am.
\newblock In {\em Festschrift for {L}ucien {L}e {C}am}, pages 423--435.
  Springer, New York, 1997.

\end{thebibliography}

\end{document}